\documentclass{article}

\usepackage[utf8]{inputenc} 
\usepackage[T1]{fontenc}    
\usepackage{booktabs}       
\usepackage{amsfonts}       
\usepackage{nicefrac}       
\usepackage{microtype}      

\usepackage{graphicx}
\usepackage{subfigure}
\usepackage{algorithm}
\usepackage{algorithmic}

\usepackage{url}

\usepackage{amsfonts}
\usepackage{amssymb}
\usepackage{amsthm}
\usepackage{amsmath}
\usepackage{url}
\usepackage{multirow}
\usepackage{color}

\usepackage[final]{nips_2018}

\usepackage{enumitem}

\newcommand{\NM}[2]{\| #1 \|_{#2}}

\newcommand{\R}{\mathbb{R}}
\def\a{\alpha}
\def\b{\beta}
\newcommand{\sign}[1]{ \text{\normalfont sign} \left( #1 \right) }

\newcommand{\Diag}[1]{\text{\normalfont Diag}( #1 )}
\newcommand{\nnz}[1]{\text{\normalfont nnz}(  #1 ) }
\newcommand{\ssum}[1]{\text{\normalfont sum}( #1 ) }

\newtheorem{assumption}{Assumption}
\newtheorem{theorem}{Theorem}[section]
\newtheorem{lemma}[theorem]{Lemma}
\newtheorem{proposition}[theorem]{Proposition}
\newtheorem{corollary}[theorem]{Corollary}

\newtheorem{remark}{Remark}[section]
\newtheorem{definition}{Definition}[section]

\newcommand{\Tr}[1]{\text{\normalfont tr}(  #1 ) }
\newcommand{\invH}[1]{\mathcal{H}_{\Omega}^{\!-\!1} (  #1 ) }

\usepackage{array}
\newcolumntype{L}[1]{>{\raggedright\let\newline\\\arraybackslash\hspace{0pt}}m{#1}}
\newcolumntype{C}[1]{>{\centering\let\newline  \\\arraybackslash\hspace{0pt}}m{#1}}
\newcolumntype{R}[1]{>{\raggedleft\let\newline \\\arraybackslash\hspace{0pt}}m{#1}}

\begin{document}

\title{Scalable Robust Matrix Factorization with Nonconvex Loss}

\author{
	Quanming Yao$^{1}$, James T. Kwok$^2$ \\
	$^1$4Paradigm Inc. Beijing, China\\
	$^2$Department of Computer Science and Engineering, \\
	Hong Kong University of Science and Technology, Hong Kong\\
	yaoquanming@4paradigm.com, jamesk@cse.ust.hk
}

\maketitle

\begin{abstract}
Matrix factorization (MF), which uses the $\ell_2$-loss, and
robust matrix factorization (RMF), which uses the $\ell_1$-loss, are
sometimes not 
robust enough for outliers.
Moreover,
even the state-of-the-art RMF solver (RMF-MM) is slow and  cannot
utilize data sparsity. 
In this paper, we propose 
to improve robustness
by using nonconvex loss  functions.
The resultant 
optimization problem
is difficult.
To improve efficiency and scalability,
we propose to use
the majorization-minimization 
(MM) 
and optimize the MM surrogate by using
the accelerated proximal gradient algorithm
on its dual problem.
Data sparsity 
can also be exploited. 
The resultant algorithm has
low time and space complexities, and is guaranteed to converge to a critical point.
Extensive experiments show that it outperforms the state-of-the-art in
terms of both accuracy and speed.
\end{abstract}



\section{Introduction}
\label{sec:intro}

Matrix factorization (MF) is a fundamental 
tool in 
machine learning,
and an important component in many applications such as computer vision
\cite{basri2007photometric,yao2018large},
social networks \cite{yang2013overlapping}
and recommender systems \cite{mnih2008probabilistic}.
The square loss has been commonly used in MF
\cite{candes2009exact,mnih2008probabilistic}.
This implicitly assumes the Gaussian noise,
and is sensitive to outliers.
Eriksson and van den Hengel
\cite{eriksson2010efficient}
proposed 
robust matrix factorization (RMF),
which uses the $\ell_1$-loss instead,
and obtains much better empirical performance.
However,  
the resultant nonconvex nonsmooth optimization problem is much more difficult.

Most RMF solvers are not scalable
\cite{cabral2013unifying,eriksson2010efficient,kim2015efficient,meng2013cyclic,zheng2012practical}.
The current state-of-the-art solver is RMF-MM \cite{lin2017robust}, which is based on majorization minimization (MM) \cite{hunter2004tutorial,lange2000optimization}.
In each iteration, 
a convex nonsmooth surrogate is 
optimized.
RMF-MM is advantageous in that it has theoretical convergence guarantees, and
demonstrates fast empirical convergence 
\cite{lin2017robust}.
However, it
cannot utilize data sparsity.
This is problematic in applications such as structure from motion
\cite{koenderink1991affine} and recommender system \cite{mnih2008probabilistic},
where the data matrices, though large, are often sparse.

Though the $\ell_1$-loss used in RMF is 
more robust than the 
$\ell_2$, still it may 
not be robust enough for outliers.  
Recently, 
better empirical performance is obtained 
in total-variation image denosing
by using
the
$\ell_0$-loss instead 
\cite{yan2013restoration}, and 
in sparse coding 
the capped-$\ell_1$ loss 
\cite{jiang2015robust}.
A similar observation is also made on the $\ell_1$-regularizer in
sparse learning and low-rank matrix learning 
\cite{gong2013general,yao2018large,zuo2013generalized}.
To alleivate this problem,
various nonconvex regularizers
have been introduced.
Examples include the Geman penalty \cite{geman1995nonlinear},
Laplace penalty \cite{trzasko2009highly},
log-sum penalty (LSP) \cite{candes2008enhancing} 
	minimax concave penalty (MCP) \cite{zhang2010nearly},
	and the smooth-capped-absolute-deviation (SCAD) penalty \cite{fan2001variable}.
These regularizers 
are similar 
in shape
to Tukey's biweight function in robust statistics \cite{huber2011robust},
	which flattens for large values.
	Empirically,
they achieve much better 
performance than $\ell_1$ on tasks such as feature selection \cite{gong2013general,zuo2013generalized}
and image denoising \cite{yao2018large}.

In this paper, we propose to improve the 
robustness of
RMF 
by using these nonconvex functions
(instead of $\ell_1$ or $\ell_2$)
as the loss function.
The resultant optimization problem is difficult, and existing RMF solvers
cannot be used.
As in RMF-MM, we rely on the more flexible MM optimization technique, and
a new MM surrogate is proposed.
To improve scalabiltiy,
we transform the surrogate to its dual and then solve  it
with  the accelerated proximal gradient (APG) algorithm \cite{beck2009fast,nesterov2013gradient}.
Data sparsity can also be exploited in the design of the APG algorithm.
As for its convergence analysis, proof techniques in RMF-MM cannot be used as the loss is 
no longer convex. 
Instead,
we develop
new proof techniques 
based on the 
Clarke subdifferential \cite{clarke1990optimization},
and show that convergence to a critical point can be guaranteed.
Extensive experiments  on both synthetic and real-world data sets
demonstrate superiority of the proposed algorithm over the state-of-the-art in terms of
both accuracy and scalability.



\noindent
{\bf Notation.}
For scalar $x$, $\sign{x} = 1$ if $x > 0$, 0 if $x  = 0$, and $-1$ otherwise.
For a vector $x$, $\Diag{x}$ constructs a diagonal matrix $X$ with $X_{ii} = x_i$.
For a matrix $X$,
$\NM{X}{F} = (\sum_{i,j} X_{ij}^2)^{1/2}$ is its Frobenius norm,
$\NM{X}{1} = \sum_{i, j} |X_{ij}|$ is its $\ell_1$-norm,
and
$\nnz{X}$ is the number of nonzero elements in $X$.
For a square matrix $X$,
$\Tr{X} = \sum_{i} X_{ii}$ is its trace.
For two matrices $X, Y$,
$\odot$ denotes element-wise product.
For a smooth function $f$,
$\nabla f$ is its gradient.
For a convex $f$,
$G \in \partial f(X) = \{ U: f(Y) \ge f(X) +  \Tr{U^\top (Y - X)} \}$ is a subgradient.


\section{Related Work}

\subsection{Majorization Minimization}

\label{sec:mm}

Majorization minimization (MM)
is a general technique to make difficult optimization problems easier
\cite{hunter2004tutorial,lange2000optimization}.
Consider 
a function
$h(X)$, which is
hard to optimize.
Let the iterate at the $k$th 
MM
iteration be $X^k$.
The next iterate 
is generated 
as
$X^{k + 1}  =  X^k  + \arg\min_{X} f^k (X; X^k)$,
where $f^k$ is a surrogate that is being optimized instead of $h$.
A good surrogate should have the following properties
\cite{lange2000optimization}:
(i) $h(X^k + X) \le f^k (X; X^k)$ for any $X$;
(ii) $0 = \arg\min_{X} \left(  f^k (X; X^k)  -  h(X^k  +  X) \right)$ and $h(X^k) = f^k(0; X^k)$; and
(iii) $f^k$ is convex on $X$.
MM only guarantees that the objectives obtained in successive
iterations are non-increasing, but
does not guarantee convergence of $X^k$
\cite{hunter2004tutorial,lange2000optimization}.




\subsection{Robust Matrix Factorization (RMF)}

\label{sec:review:rmf}

In matrix factorization (MF), the data matrix $M \in \R^{m \times n}$ is approximated by
$U V^{\top}$,
where $U \in \R^{m \times r}$, $V \in \R^{n \times r}$ and $r \ll \min( m, n )$ is
the rank.
In applications such as structure from motion (SfM) \cite{basri2007photometric} and recommender systems \cite{mnih2008probabilistic},
some entries of $M$ may be missing.
In general,
the MF problem can 
be formulated as:
$\min_{U,V}
\frac{1}{2} \NM{W  \odot  ( M  -  U V^{\top} ) }{F}^2
+ \frac{\lambda}{2}( \NM{U}{F}^2 + \NM{V}{F}^2 )$,
where 
$W \in \{ 0, 1 \}^{m \times n}$ 
contain indices to the observed entries in $M$ (with $W_{ij} = 1$ if $M_{ij}$ is observed, and 0 otherwise),
and 
$\lambda \ge 0$ is a regularization parameter.
The $\ell_2$-loss 
is sensitive to outliers.  
In \cite{de2003framework},
it is replaced by the $\ell_1$-loss,
leading to robust matrix factorization (RMF):
\begin{equation}
\min_{U,V} 
\NM{W \odot ( M - U V^{\top} ) }{1}
+ \frac{\lambda}{2} ( \NM{U}{F}^2 + \NM{V}{F}^2 ).
\label{eq:rmf}
\end{equation}
Many RMF solvers have been developed \cite{cambier2016robust,eriksson2010efficient,he2012incremental,cabral2013unifying,kim2015efficient,lin2017robust,meng2013cyclic,zheng2012practical}.
However, as the objective in \eqref{eq:rmf} is neither convex nor smooth,
these solvers lack scalability,
robustness
and/or convergence guarantees.
Interested readers are referred to Section~2 of \cite{lin2017robust} for details.

Recently,
the RMF-MM algorithm \cite{lin2017robust}
solves \eqref{eq:rmf} using MM.  Let the $k$th iterate be $(U^k,
V^k)$.  RMF-MM  tries to find increments $(\bar{U}, \bar{V})$ that should be added
to  obtain
the target $(U, V)$:
\begin{equation} 
\label{eq:uv}
U=U^k+ \bar{U},
\quad
V = V^k + \bar{V}.
\end{equation}
Substituting into (\ref{eq:rmf}),
the objective 
can be rewritten as
$H^k(  \bar{U}, \bar{V} ) 
\equiv \NM{W \odot ( M \! - \! ( U^k + \bar{U} ) ( V^k + \bar{V} )^{\top} )}{1}
+ \frac{\lambda}{2} \NM{U^k + \bar{U}}{F}^2
+ \frac{\lambda}{2} \NM{V^k + \bar{V}}{F}^2$.
The following Proposition constructs a surrogate 
$F^k$ 
of $H^k$ 
that 
satisfies properties (i) and (ii) 
in Section~\ref{sec:mm}.
Unlike $H^k$, 
$F^k$ is jointly convex in $(\bar{U}, \bar{V})$.


\begin{proposition} \cite{lin2017robust}
\label{pr:rmf:surr}
Let $\nnz{ W_{(i,:)} }$ (resp. $\nnz{W_{(:,j)}}$)
be the number of nonzero elements in the $i$th row (resp. $j$th column) of $W$, 
$\Lambda_r
= \text{Diag}
(\sqrt{\smash[b]{ \nnz{ W_{(1,:)} } }}, \dots, \sqrt{\smash[b]{ \nnz{W_{(m,:)} }} })$,
and $\Lambda_c  = \text{Diag}(  \sqrt{\smash[b]{  \nnz{ W_{(:,1)}} } }, \dots, \sqrt{\smash[b]{
		\nnz{ W_{(:, n)} }}})$.
Then,
$H^k ( \bar{U}, \bar{V} ) \le F^k( \bar{U}, \bar{V} )$,
where
\begin{align}
F^k(\bar{U}, \bar{V}) 
\equiv
& \NM{W \! \odot \! ( M - U^k (V^k)^{\top} - 
\bar{U} (V^k)^{\top} - U^k \bar{V}^{\top} )}{1} 
\notag
\\
& + \frac{\lambda}{2} \NM{U^k + \bar{U}}{F}^2 + \frac{1}{2} \NM{\Lambda_r \bar{U}}{F}^2
+ \frac{\lambda}{2} \NM{V^k + \bar{V}}{F}^2 + \frac{1}{2} \NM{\Lambda_c \bar{V}}{F}^2.
\label{eq:surr}
\end{align}
Equality holds iff $( \bar{U}, \bar{V} ) = (0, 0)$.
\end{proposition}

Because of the coupling of 
$\bar{U}, V^k$ 
(resp.  $U^k, \bar{V}$)
in $\bar{U} (V^k)^{\top}$ (resp. $U^k \bar{V}^{\top}$) in 
(\ref{eq:surr}),
$F^k$
is still difficult to optimize.
To address this problem, 
RMF-MM 
uses 
the 
LADMPSAP algorithm \cite{lin2015linearized}, which is
a multi-block 
variant 
of the alternating direction method of multipliers (ADMM)
\cite{boyd2011distributed}.

RMF-MM has a space complexity of $O(m n)$, 
and a time complexity of $O( m n r I K )$,  
where $I$ is the number of (inner) LADMPSAP iterations and 
$K$ is the number of (outer) RMF-MM iterations.
These grow linearly with the matrix size, and can be expensive on large data sets.
Besides, as discussed in Section~\ref{sec:intro}, the $\ell_1$-loss may still be sensitive to outliers.


\section{Proposed Algorithm}


\subsection{Use a More Robust Nonconvex Loss}

\label{sec:penalty}

In this paper, we improve robustness of RMF by using a general nonconvex loss instead of the $\ell_1$-loss.
Problem~(\ref{eq:rmf}) is then changed to:
\begin{align}
\min_{U,V} \dot{H}(U, V)
\equiv
\sum_{i = 1}^m \sum_{j = 1}^n
W_{ij} \phi\left( |  M_{ij} - [  U V^{\top} ]_{ij} | \right) 
+ \frac{\lambda}{2}(  \NM{U}{F}^2 + \NM{V}{F}^2 ) ,
\label{eq:mrmf}
\end{align}
where $\phi$ is nonconvex.
We assume the following on $\phi$:
\begin{assumption} \label{ass:phi}
$\phi(\alpha)$ is concave,
smooth and 
strictly increasing on $\alpha \ge 0$.
\end{assumption}

Assumption~\ref{ass:phi} is satisfied by many nonconvex functions,
including the Geman, Laplace and LSP penalties mentioned in
Section~\ref{sec:intro}, and slightly modified variants of the MCP
	and SCAD penalties.
	Details can be found in Appendix~\ref{app:modmcp}.
Unlike previous papers
\cite{gong2013general,zuo2013generalized,yao2018large}, we use
these nonconvex functions 
as the \textit{loss},
not as regularizer.
The $\ell_1$ 
also satisfies Assumption~\ref{ass:phi}, and thus (\ref{eq:mrmf}) includes (\ref{eq:rmf}).

When the $i$th row of $W$ is zero, the $i$th row of $U$ obtained 
is zero because of the
$\NM{U}{F}^2$ regularizer.
Similarly,
when the $i$th column of $W$ is zero,
the corresponding column in $V$ is zero.
To avoid this trivial solution, 
we make the following Assumption, which is also used in 
matrix completion  
\cite{candes2009exact} and RMF-MM.
\begin{assumption} \label{ass:wht}
	$W$ has no zero row or column.
\end{assumption}


\subsection{Constructing the Surrogate}

\label{sec:newsurr}

Problem (\ref{eq:mrmf})
is difficult to solve, and
existing RMF solvers cannot be used as they rely crucially on the $\ell_1$-norm.
In this Section, we use the more flexible MM technique as in 
RMF-MM.
However, 
its surrogate construction scheme 
cannot be used here.
RMF-MM 
uses
the 
convex
$\ell_1$ loss,
and only needs to handle nonconvexity resulting from the product $U V^{\top}$
in \eqref{eq:rmf}.
Here,
nonconvexity 
in (\ref{eq:mrmf})
comes from both from the loss and 
$U V^{\top}$.




The following Proposition 
first obtains a convex upper bound of the nonconvex $\phi$
using Taylor expansion. 
An illustration 
is shown in Figure~\ref{fig:phibnd}. 
Note that this upper bound is simply a re-weighted $\ell_1$, with scaling factor
$\phi'(|\b|)$ and offset $\phi(|\b|) - \phi'(|\b|) |\b|$.
As one may expect, recovery of the $\ell_1$ makes optimization easier.
It is known that the LSP, when used as a regularizer, can be interpreted as re-weighted $\ell_1$ regularization \cite{candes2009exact}.
Thus, Proposition~\ref{pr:bndphi} includes this as a special case.

\begin{proposition} 
\label{pr:bndphi}
For any given $\beta \in \R$,
$\phi(|\a|) \le \phi'(|\b|) | \a| + (  \phi(|\b|) - \phi'(|\b|) |\b| )$,
and the equality holds iff $\a= \pm \b$.
\end{proposition}

\begin{figure}[ht]
\centering
\subfigure[Geman.]
{\includegraphics[width=0.32\columnwidth]{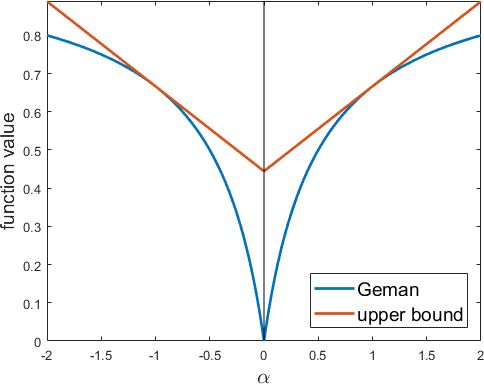}}
\subfigure[Laplace.]
{\includegraphics[width=0.32\columnwidth]{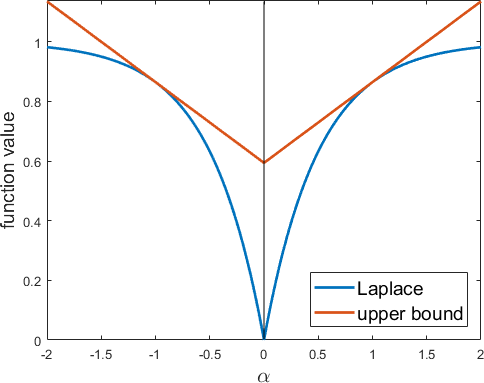}}
\subfigure[LSP.]
{\includegraphics[width=0.32\columnwidth]{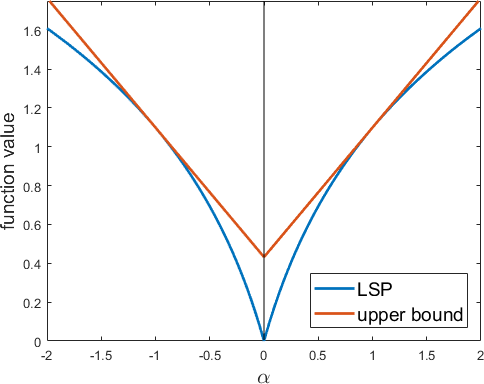}}
\subfigure[modified MCP.]
{\includegraphics[width=0.32\columnwidth]{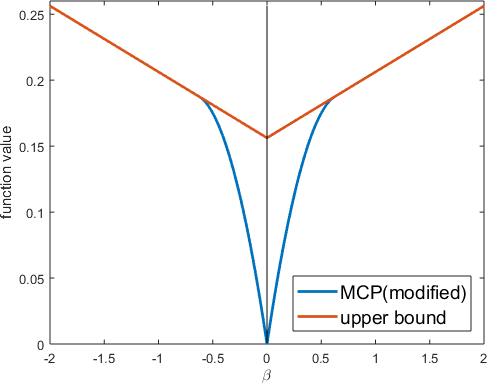}
	\label{fig:mcp}}
\subfigure[modified SCAD.]
{\includegraphics[width=0.32\columnwidth]{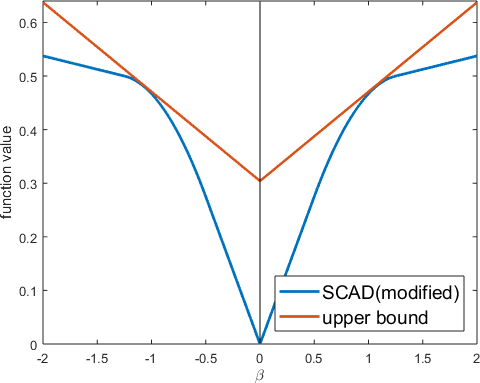}
	\label{fig:scad}}
\caption{Upper bounds
	for the various nonconvex penalities (see Table~\ref{tab:regdef} in Appendix~\ref{sec:defncvxf})
	$\beta = 1$, $\theta = 2.5$ for SCAD and $\theta = 0.5$ for the others;
	and $\delta = 0.05$ for MCP and SCAD.}
\label{fig:phibnd}
\end{figure}


Given the current iterate $(U^k, V^k)$, we want to find increments $(\bar{U},
\bar{V})$ as in \eqref{eq:uv}.
$\dot{H}$ in  \eqref{eq:mrmf}  can be rewritten as:
$\dot{H}^k( \bar{U}, \bar{V} ) 
\equiv \sum_{i = 1}^m \sum_{j = 1}^n
W_{ij} 
\phi( | M_{ij}  - [ (U^k \! + \bar{U}) (V^k + \bar{V})^{\top} ]_{ij} | )
+ \frac{\lambda}{2} \NM{U^k + \bar{U}}{F}^2
+ \frac{\lambda}{2} \NM{V^k + \bar{V}}{F}^2$.
Using Proposition~\ref{pr:bndphi},
we obtain the following convex upper bound for 
$\dot{H}^k$.

\begin{corollary} 
\label{cor:loss}
$\dot{H}^k (\bar{U}, \bar{V})
\le 
b^k
+ \frac{\lambda}{2} \NM{U^k + \bar{U}}{F}^2
+ \frac{\lambda}{2} \NM{V^k + \bar{V}}{F}^2 
+ \|\dot{W}^k \odot ( M - U^k (V^k)^{\top} 
-  \bar{U} (V^k)^{\top} 
-  U^k \bar{V}^{\top} 
-  \bar{U} \bar{V}^{\top} ) \|_1$, 
where $b^k = \sum_{i = 1}^m \sum_{j = 1}^n W_{ij} ( \phi( |  [ U^k (V^k)^{\top} ]_{ij} | )
- A^k_{ij} | [ U^k (V^k)^{\top} ]_{ij} | )$, $\dot{W}^k$ $= A^k \odot W$,
and $A^k_{ij} = \phi'( | [U^k (V^k)^{\top}]_{ij} | )$.
\end{corollary}

The product $\bar{U} \bar{V}^{\top}$ 
still couples $\bar{U}$ and $\bar{V}$ together. 
As $\dot{H}^k$ is similar to $H^k$ in 
Section~\ref{sec:review:rmf},
one may want to reuse Proposition~\ref{pr:rmf:surr}.
However, 
Proposition~\ref{pr:rmf:surr} holds only when $W$ is a binary matrix,
while $\dot{W}^k$ here is real-valued.
Let 
$\Lambda_r^k = \Diag{\sqrt{\smash[b]{\ssum{\dot{W}_{(1, :)}^k}}}, \dots ,
	\sqrt{\smash[b]{\ssum{\dot{W}_{(m, :)}^k}}}}$
and
$\Lambda_c^k  = \Diag{\sqrt{\smash[b]{\ssum{\dot{W}_{(:, 1)}^k}}}, \dots, \sqrt{\smash[b]{\ssum{\dot{W}_{(:, n)}^k}} }}$.
The following Proposition shows that
$\dot{F}^k ( \bar{U}, \bar{V} ) 
\equiv 
\NM{\dot{W}^k \odot ( M - U^k (V^k)^{\top} - \bar{U} (V^k)^{\top} 
	- U^k \bar{V}^{\top} )}{1}
+ \frac{\lambda}{2} \NM{U^k + \bar{U}}{F}^2
+ \frac{1}{2} \NM{\Lambda_r^k \bar{U}}{F}^2
+ \frac{\lambda}{2} \NM{V^k + \bar{V}}{F}^2
+ \frac{1}{2} \NM{\Lambda_c^k \bar{V}}{F}^2
+ b^k$,
can be used as a surrogate.
Moreover, it can be 
easily 
seen that $\dot{F}^k$ qualifies as
a good surrogate
in Section~\ref{sec:mm}:
(a) $\dot{H}(\bar{U} + U^k, \bar{V} + V^k) 
\le \dot{F}^k(\bar{U}, \bar{V})$;
(b) $(0, 0) = \arg\min_{\bar{U}, \bar{V}} \dot{F}^k(\bar{U}, \bar{V}) - \dot{H}^k(\bar{U}, \bar{V})$ 
and $\dot{F}^k ( 0, 0 ) = \dot{H}(0, 0)$; and
(c) $\dot{F}^k$ is jointly convex in $\bar{U}, \bar{V}$.

\begin{proposition}
\label{pr:bndext}
$\dot{H}^k ( \bar{U}, \bar{V} ) \le \dot{F}^k ( \bar{U}, \bar{V} )$, with
equality holds iff $( \bar{U}, \bar{V} ) = (0, 0)$.
\end{proposition}

\begin{remark}
In the special case where the $\ell_1$-loss is used,
$\dot{W}^k = W$, $b^k = 0$ 
$\Lambda_r^k = \Lambda_r$, and $\Lambda_c^k = \Lambda_c$.  The surrogate 
$\dot{F}^k ( \bar{U}, \bar{V} )$
then reduces to that in \eqref{eq:surr},
and Proposition~\ref{pr:bndext} becomes Proposition~\ref{pr:rmf:surr}.
\end{remark}






\subsection{Optimizing the Surrogate via APG on the Dual}

\label{sec:subsolve}


LADMPSAP, which is used in RMF-MM, can also be used to optimize 
$\dot{F}^k$.
However, 
the dual variable in LADMPSAP is a dense matrix, and 
cannot utilize possible sparsity of $W$.
Moreover, 
LADMPSAP converges at a rate  of $O(1/T)$ \cite{lin2015linearized},
which is slow.
In the following, we propose
a time- and space-efficient optimization procedure basesd on  
running the 
accelerated proximal gradient 
(APG) algorithm 
on the surrogate optimization problem's  dual. 
Note that 
while the primal problem has $O(m n)$
variables, 
the dual problem has only 
$\nnz{W}$ 
variables.




\subsubsection{Problem Reformulation}
Let $\Omega \equiv \{( i_1, j_1 ), \dots, ( i_{\nnz{W}}, j_{\nnz{W}} ) \}$
be the set containing
indices of the observed elements in
$W$,
$\mathcal{H}_{\Omega}(\cdot)$ be
the linear operator which maps a $\nnz{W}$-dimensional vector $x$ 
to the sparse matrix $X \in \R^{m \times n}$ with nonzero positions indicated by $\Omega$
(i.e., 
$X_{i_t j_t} = x_t$ where $(i_t, j_t)$ is the $t$th element in $\Omega$), and
$\mathcal{H}_{\Omega}^{-1}(\cdot)$ be the inverse operator of $\mathcal{H}_{\Omega}$.

\begin{proposition}
\label{pr:dual}
The dual problem of
$\min_{\bar{U}, \bar{V}} \dot{F}^k(
\bar{U}, \bar{V})$
is
\begin{eqnarray}
\min_{x \in \mathcal{W}^k} \mathcal{D}^k(x) & \equiv
&
\frac{1}{2} \Tr{ ( \mathcal{H}_{\Omega}(x) V^k - \lambda U^k )^{\top} 
A_r^k ( \mathcal{H}_{\Omega}(x) V^k - \lambda U^k ) } - \Tr{\mathcal{H}_{\Omega}(x)^{\top} M} 
\notag
\\
&& + \frac{1}{2} \Tr{(  \mathcal{H}_{\Omega}(x)^{\top} U^k  
	- \lambda V^k )^{\top} A_c^k ( \mathcal{H}_{\Omega}(x)^{\top}  U^k - 
	\lambda V^k )},
\label{eq:dual}
\end{eqnarray}
where 
$\mathcal{W}^k \equiv \{ x \in \R^{\nnz{W}} : | x_i | \le [ \dot{w}^k ]_{i}^{-1} \}$,
$\dot{w}^k = \mathcal{H}_{\Omega}^{-1}( \dot{W}^k )$,
$A_r^k = ( \lambda I + (\Lambda_r^k)^2 )^{-1}$,
and
$A_c^k = ( \lambda I + (\Lambda_c^k)^2 )^{-1}$.
From the obtained $x$, the primal 
$(\bar{U},\bar{V} )$
solution 
can be recovered as
$\bar{U} = A_r^k ( \mathcal{H}_{\Omega}(x) V^k - \lambda U^k )$
and
$\bar{V} = A_c^k (  \mathcal{H}_{\Omega}(x)^{\top} U^k - \lambda V^k )$.
\end{proposition}
Problem~\eqref{eq:dual} 
can be solved
by 
the APG algorithm,
which has 
a 
convergence rate  of
$O(1/T^2)$ 
\cite{beck2009fast,nesterov2013gradient} and is
faster
than 
LADMPSAP.
As $\mathcal{W}^k$ involves only
$\ell_1$ constraints,
the proximal step can be easily computed with closed-form (details are in Appendix~\ref{app:prox})
and takes only $O( \nnz{W} )$ time.

The complete procedure, which will be called 
Robust Matrix Factorization with Nonconvex Loss (RMFNL) algorithm
is shown in Algorithm~\ref{alg:rmfnl}.
The surrogate is optimized via its dual in step~4.
The primal solution is recovered in step~5,
and $(U^k, V^k)$ are updated in step~6.

\begin{algorithm}[ht]
\caption{Robust matrix factorization using nonconvex loss (RMFNL) algorithm.}
\begin{algorithmic}[1]
\STATE initialize $U^1 \in \R^{m \times r}$ and $V^1 \in \R^{m \times r}$;
\FOR{$k = 1, 2, \dots, K$}

\STATE compute $\dot{W}^k$ 
in Corollary~\ref{cor:loss}
(only on the observed positions), and
$\Lambda_r^k, \Lambda_c^k$;
	
\STATE compute $x^k = \arg\min_{x \in \mathcal{W}^k} \mathcal{D}^k(x)$ in
Proposition~\ref{pr:dual} using APG;
	
\STATE $\bar{U}^k = A_r^k \left(  \mathcal{H}_{\Omega}(x^k) V^k - \lambda U^k \right)$, \;
$\bar{V}^k = A_c^k (  \mathcal{H}_{\Omega}( x^k )^{\top} U^k - \lambda V^k )$;
	
\STATE $U^{k + 1} = U^k + \bar{U}^k$, \; $V^{k + 1} = V^k + \bar{V}^k$;

\ENDFOR
\RETURN $U^{K + 1}$ and $V^{K + 1}$.
\end{algorithmic}
\label{alg:rmfnl}
\end{algorithm}

\subsubsection{Exploiting Sparsity}
A direct implementation of APG takes 
$O( mn )$ space
and 
$O( m n r)$ time  per iteration.
In the following, we show how 
these 
can be reduced
by exploiting sparsity fo $W$.

The objective in \eqref{eq:dual} involves $A_r^k, A_c^k$ and $\mathcal{W}^k$, which are all related to 
$\dot{W}^k$. 
Recall that
$\dot{W}^k$ 
in Corollary~\ref{cor:loss}
is sparse (as $W$ is sparse).
Thus, by exploting sparsity, 
constructing $A_r^k, A_c^k$ and $\mathcal{W}^k$ only take
$O(\nnz{W})$ time and space.

In each APG iteration,
one has to compute the 
gradient,
objective, 
and proximal step.
First, 
consider the gradient 
$\nabla \mathcal{D}^k(x)$ of the objective, which is equal to
\begin{eqnarray}
\invH{ A_r^k ( \mathcal{H}_{\Omega}(x) V^k - \lambda U^k ) (V^k)^{\top} }
+ \invH{U^k [  (U^k)^{\top} \mathcal{H}_{\Omega}(x) 
	- \lambda (V^k)^{\top} ] A_c^k} - \invH{ M }.
\!\!\!\!
\label{eq:grad}
\end{eqnarray}
The first term  can be rewritten as $\hat{g}^k = \invH{ Q^k (V^k)^{\top} }$, 
where $Q^k = A_r^k (  \mathcal{H}_{\Omega}(x) V^k - \lambda U^k ) $.
As $A_r^k$ is diagonal and $\mathcal{H}_{\Omega}(x)$ is sparse,
$Q^k$ can be computed as
$A_r^k ( \mathcal{H}_{\Omega}(x) V^k )  - \lambda ( A_r^k U^k )$ in $O( \nnz{W} r + m r )$ time,
where $r$ is the number of columns in $U^k$ and $V^k$.
Let the $t$th element in $\Omega$ be $(i_t, j_t)$.
By the definition of $\invH{\cdot}$,
we have $\hat{g}^k_t = \sum_{q = 1}^r Q^k_{i_t q} V^k_{j_t q}$, 
and this takes $O( \nnz{W} r + m r )$ time.
Similarly,
computing the second term in \eqref{eq:grad} takes $O( \nnz{W} r + n r )$ time.
Hence,
computing 
$\nabla \mathcal{D}^k(x)$ takes a total of $O(  \nnz{W} r + (m + n) r )$ time and $O(  \nnz{W} + (m + n) r )$ space
(the Algorithm
is shown 
in Appendix~\ref{app:compgrad}).
Similarly,
the objective can be obtained
in $O(  \nnz{W} r + (m + n) r )$ time
and $O(  \nnz{W} + (m + n) r )$ space (details are in Appendix~\ref{app:compobj}).
The proximal step takes $O(\nnz{W})$ time and space, as $x \in \R^{\nnz{W}}$.
Thus,
by exploiting sparsity,
the APG algorithm has a
space complexity of
$O(  \nnz{W} + (m + n) r )$
and iteration time complexity of $O(  \nnz{W} r + (m + n) r )$.
In comparison,
LADMPSAP needs $O( mn )$ space and iteration time complexity of 
$O( m n r )$.
A summary of the complexity results is shown in Figure~\ref{fig:compvar}.




\subsection{Convergence Analysis}

\label{sec:analysis}

In this section, we study the convergence of 
RMFNL.
Note that the proof technique in RMF-MM cannot be used, as it relies on convexity of
the $\ell_1$-loss while
$\phi$ in (\ref{eq:mrmf}) is nonconvex
(in particular, Proposition 1 in \cite{lin2017robust} fails). 
Moreover,
the proof of RMF-MM 
uses 
the subgradient.
Here,
as $\phi$ is  
nonconvex,
we will use the 
Clarke subdifferential
\cite{clarke1990optimization},
which 
generalizes
subgradients to nonconvex functions
(a brief introduction is in
Appendix~\ref{app:clarke}).
For the iterates $\{ X^k \}$ 
generated by RMF-MM,
it is guaranteed to have  a \textit{sufficient decrease}
on the objective 
$f$ 
in the following sense
\cite{lin2017robust}:
There exists a constant $\gamma > 0$ such that
$f(X^k) - f(X^{k + 1})
\ge \gamma \NM{X^k - X^{k + 1}}{F}^2, \forall k$.
The following Proposition shows that  RMFNL also achieves a sufficient decrease on
its objective. Moreover, the $\{(U^k ,V^k )\}$ sequence generated is bounded,
which has at least one limit point.

\begin{proposition}
\label{pr:limit}
For Algorithm~\ref{alg:rmfnl}, 
$\{( U^k, V^k ) \}$ is bounded,  and
has a sufficient decrease on $\dot{H}$.
\end{proposition}


\begin{theorem}
	\label{thm:critical}
	The limit points 
	of the sequence generated by Algorithm~\ref{alg:rmfnl}
	are critical points of \eqref{eq:mrmf}.
\end{theorem}


\section{Experiments}

\label{sec:exp}

In this section, we compare the proposed RMFNL with state-of-the-art
MF algorithms.
Experiments are performed on a PC with Intel i7 CPU and 32GB RAM.  
All the codes are in Matlab,
with sparse matrix operations implemented in C++.
We use the nonconvex loss functions
of LSP, Geman and Laplace in
Table~\ref{tab:regdef} of Appendix~\ref{app:modmcp},
with $\theta=1$;
and fix $\lambda = 20/(m + n)$ in (\ref{eq:rmf})
as suggested in \cite{lin2017robust}.

\subsection{Synthetic Data}

\label{sec:expsyn}

We first perform experiments on synthetic data, which
is generated as $X = U V^{\top}$ with $U \in \R^{m \times 5}$, $V \in \R^{m \times 5}$,
and $m=\{250, 500, 1000\}$.
Elements of $U$ and $V$ are sampled i.i.d. from the standard normal distribution $\mathcal{N}(0, 1)$.
This is then corrupted to form $M = X + N + S$,
where 
$N$ is the noise matrix from $\mathcal{N}(0, 0.1)$, and
$S$ is a sparse matrix modeling outliers
with 5\% nonzero elements randomly sampled from $\{ \pm 5 \}$.
We randomly draw $10\log(m)/m\%$ of the elements from $M$ as observations, 
with half of them for training and the other half for validation.
The remaining unobserved elements are for testing.
Note that the 
larger the 
$m$,
the sparser 
is the observed matrix.

The iterate $(U^1, V^1)$ 
is initialized
as Gaussian random matrices,
and the iterative procedure is stopped when the relative change in objective values between successive iterations
is smaller than $10^{-4}$.
For the subproblems in RMF-MM and RMFNL,
iteration is stopped when the relative change in objective value is smaller than
$10^{-6}$ or when a maximum of $300$ iterations is used.
The rank $r$ is set to the ground truth (i.e., 5).
For performance evaluation, we follow \cite{lin2017robust} and use the (i) testing
root mean square error,
$\text{RMSE} 
= \sqrt{\smash[b]{\NM{\bar{W} \odot (X - \bar{U} \bar{V}^T)}{F}^2 / \nnz{\bar{W}}}}$,
where 
$\bar{W}$
is a binary matrix 
indicating positions of the testing elements;
and (ii) CPU time.
To reduce statistical variability, results are averaged over five repetitions.

\subsubsection{Solvers for Surrogate Optimization}
Here, we compare three solvers for
surrogate  optimization in each  
RMFNL iteration
(with the LSP loss and $m = 1000$):
(i) 
LADMPSAP
in RMF-MM;
(ii) 
APG(dense), which 
uses APG but without 
utilizing data sparsity;
and 
(iii) 
APG in Algorithm~\ref{alg:rmfnl}, 
which utilizes data sparsity as in Section \ref{sec:subsolve}.
The APG stepsize is determined by line-search,
and adaptive restart is used for further speedup \cite{nesterov2013gradient}.
Figure~\ref{fig:surrogate}
shows convergence in the first RMFNL iteration
(results for the other iterations are similar).
As can be seen,
LADMPSAP is the slowest w.r.t. the number of iterations, 
as its convergence rate is inferior to both variants of APG (whose rates are the same).
In terms of CPU time,
APG is the fastest as it can also utilize data sparsity.


\begin{figure}[ht]
\centering
\subfigure[Complexities of surrogate optimizers.]
{\includegraphics[width=0.40\columnwidth]{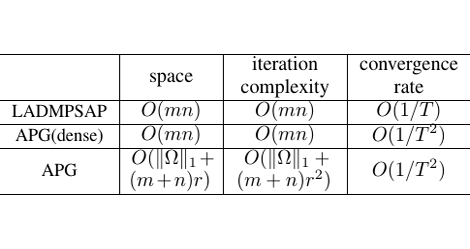}
\label{fig:compvar}}
\subfigure[Number of iterations. \label{fig:iter}]
{\includegraphics[width=0.28\columnwidth]{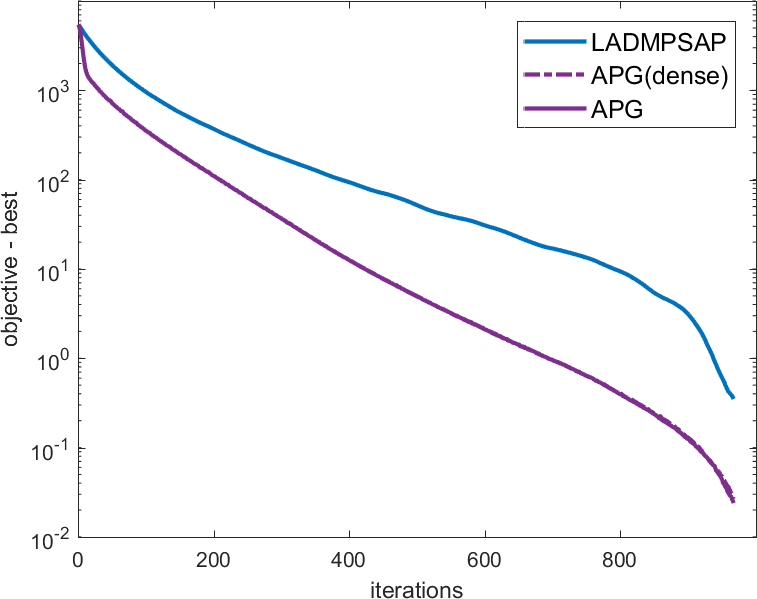}}
\subfigure[CPU time.]
{\includegraphics[width=0.28\columnwidth]{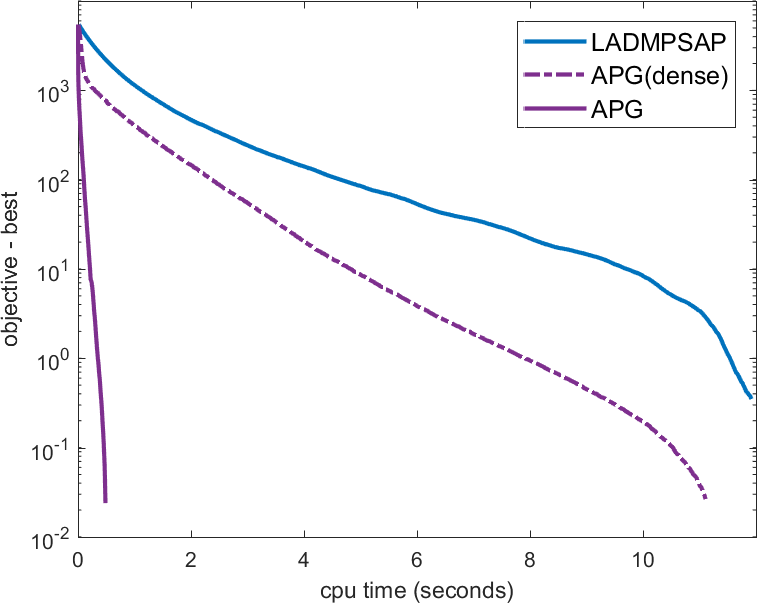}}

\caption{Convergence of the objective on the synthetic data set (with the LSP loss
and $m = 1000$).
Note that the curves for APG-dense and APG overlap in Figure~\ref{fig:iter}.}
\label{fig:surrogate}
\end{figure}

Table~\ref{tab:variant} 
shows performance of the whole RMFNL algorithm with 
different surrogate optimizers.\footnote{For all tables in the sequel, the best and comparable results according to the pairwise t-test with $95\%$ confidence are highlighted.}
As can be seen,
the various nonconvex losses 
(LSP, Geman and Laplace)
lead to similar RMSE's, as
has been similarly observed in
\cite{gong2013general,yao2018large}.
Moreover,
the different optimizers all obtain
the same RMSE.
In terms of speed,
APG is the fastest, then followed by
APG(dense), and
LADMPSAP 
is the slowest.
Hence, in the sequel, we will only use APG to
optimize
the surrogate.

\begin{table}[ht]
\centering
\caption{Performance of RMFNL with different surrogate optimizers.}
\scalebox{0.825}{
\begin{tabular}{c | c | c | c | c | c | c | c }
	\hline
      &            &   \multicolumn{2}{c|}{$m=250$ (nnz: 11.04\%)}   &   \multicolumn{2}{c|}{$m=500$ (nnz: 6.21\%)}    &   \multicolumn{2}{c}{$m=1000$ (nnz: 3.45\%)}    \\
loss   & solver & RMSE                     & CPU time             & RMSE                     & CPU time             & RMSE                    & CPU time              \\ \hline
      & LADMPSAP   & \textbf{0.110$\pm$0.004} & 17.0$\pm$1.4         & \textbf{0.072$\pm$0.001} & 195.7$\pm$34.7       & \textbf{0.45$\pm$0.007} & 950.8$\pm$138.8       \\ \cline{2-8}
 LSP   & APG(dense) & \textbf{0.110$\pm$0.004} & 12.1$\pm$0.6         & \textbf{0.073$\pm$0.001} & 114.4$\pm$18.8       & \textbf{0.45$\pm$0.007} & 490.1$\pm$91.9        \\ \cline{2-8}
      & APG        & \textbf{0.110$\pm$0.004} & \textbf{3.2$\pm$0.6} & \textbf{0.073$\pm$0.001} & \textbf{5.5$\pm$1.0} & \textbf{0.45$\pm$0.006} & \textbf{24.6$\pm$3.2} \\ \hline
      & LADMPSAP   & 0.115$\pm$0.014          & 20.4$\pm$0.8         & \textbf{0.074$\pm$0.006} & 231.0$\pm$36.9       & \textbf{0.45$\pm$0.007} & 950.8$\pm$138.8       \\ \cline{2-8}
 Geman  & APG(dense) & 0.115$\pm$0.011          & 13.9$\pm$1.6         & \textbf{0.073$\pm$0.002} & 146.9$\pm$24.8       & \textbf{0.45$\pm$0.007} & 490.1$\pm$91.9        \\ \cline{2-8}
      & APG        & 0.114$\pm$0.009          & \textbf{3.1$\pm$0.5} & \textbf{0.073$\pm$0.002} & \textbf{8.3$\pm$1.1} & \textbf{0.45$\pm$0.006} & \textbf{24.6$\pm$3.2} \\ \hline
      & LADMPSAP   & \textbf{0.110$\pm$0.004} & 17.1$\pm$1.5         & \textbf{0.072$\pm$0.001} & 203.4$\pm$22.7       & \textbf{0.45$\pm$0.007} & 950.8$\pm$138.8       \\ \cline{2-8}
	Laplace & APG(dense) & \textbf{0.110$\pm$0.004} & 12.1$\pm$2.1         & \textbf{0.073$\pm$0.003} & 120.9$\pm$28.9       & \textbf{0.45$\pm$0.007} & 490.1$\pm$91.9        \\ \cline{2-8}
      & APG        & \textbf{0.111$\pm$0.004} & \textbf{2.8$\pm$0.4} & \textbf{0.074$\pm$0.001} & \textbf{5.6$\pm$1.0} & \textbf{0.45$\pm$0.006} & \textbf{24.6$\pm$3.2} \\ \hline
\end{tabular}
}
\label{tab:variant}
\end{table}

%
%

\label{sec:otheralgs}
\subsubsection{Comparison with State-of-the-Art Matrix Factorization Algorithms}

Next,
we compare RMFNL with 
state-of-the-art 
MF and RMF algorithms.
The $\ell_2$-loss-based MF algorithms
that will be compared
include
alternating gradient descent  ({AltGrad}) \cite{mnih2008probabilistic},
Riemannian preconditioning (RP)
\cite{mishra2016riemannian}, 
scaled alternating steepest descent (ScaledASD)
\cite{tanner2016low},
alternative minimization for large scale matrix imputing (ALT-Impute)
\cite{hastie2015matrix}
and online massive dictionary learning (OMDL) \cite{mensch2016dictionary}.
The $\ell_1$-loss-based RMF algorithms being compared
include
RMF-MM
\cite{lin2017robust},
robust matrix completion (RMC)
\cite{cambier2016robust}
and 
Grassmannian robust adaptive subspace tracking algorithm (GRASTA)
\cite{he2012incremental}.
Codes are provided by the respective authors.
We do not compare
with AOPMC \cite{yan2013exact}, which has been shown to be 
slower than RMC \cite{cambier2016robust}.


As can be seen from
Table~\ref{tab:synall},
RMFNL produces much lower RMSE than the  MF/RMF algorithms,
and the RMSEs from different nonconvex losses are similar.
AltGrad, RP, ScaledASD, ALT-Impute and OMDL
are very fast because they use 
the simple $\ell_2$ loss.  However, their RMSEs are much higher than RMFNL and RMF algorithms.
	A more detailed convergence comparison is shown in Figure~\ref{fig:syncompall}.
As can be seen, 
RMF-MM is the slowest.
RMFNL with different nonconvex losses have similar convergence
behavior, and they all converge 
to a lower testing RMSE
much faster than the others.

\begin{table}[ht]
\centering
\caption{Performance of the various matrix factorization algorithms on synthetic data.}
\scalebox{0.825}{
\begin{tabular}{ c| c | c | c | c | c | c | c }
	\hline
	         &            &   \multicolumn{2}{c|}{$m=250$ (nnz: 11.04\%)}   &   \multicolumn{2}{c|}{$m=500$ (nnz: 6.21\%)}    &   \multicolumn{2}{c}{$m=1000$ (nnz: 3.45\%)}    \\
	loss     & algorithm  & RMSE                     & CPU time             & RMSE                     & CPU time             & RMSE                     & CPU time             \\ \hline
	$\ell_2$ & AltGrad    & 1.062$\pm$0.040          & 1.0$\pm$0.6          & 0.950$\pm$0.005          & 1.8$\pm$0.3          & 0.853$\pm$0.010          & 6.0$\pm$4.2          \\ \cline{2-8}
	         & RP         & 1.048$\pm$0.071          & \textbf{0.1$\pm$0.1} & 0.953$\pm$0.012          & 0.4$\pm$0.2          & 0.848$\pm$0.009          & 1.1$\pm$0.1          \\ \cline{2-8}
	         & ScaledASD  & 1.042$\pm$0.066          & 0.2$\pm$0.1          & 0.950$\pm$0.009          & 0.4$\pm$0.3          & 0.847$\pm$0.009          & 1.2$\pm$0.5          \\ \cline{2-8}
	         & ALT-Impute & 1.030$\pm$0.060          & 0.2$\pm$0.1          & 0.937$\pm$0.010          & 0.3$\pm$0.1          & 0.838$\pm$0.009          & 1.0$\pm$0.2          \\ \cline{2-8}
	         & OMDL       & 1.089$\pm$0.055          & \textbf{0.1$\pm$0.1} & 0.945$\pm$0.018          & \textbf{0.2$\pm$0.1} & 0.847$\pm$0.009          & \textbf{0.5$\pm$0.2} \\ \hline
	$\ell_1$ & GRASTA     & 0.338$\pm$0.033          & 1.5$\pm$0.1          & 0.306$\pm$0.002          & 2.9$\pm$0.3          & 0.244$\pm$0.009          & 6.1$\pm$0.4          \\ \cline{2-8}
	         & RMC        & 0.226$\pm$0.040          & 2.8$\pm$1.0          & 0.201$\pm$0.001          & 2.7$\pm$0.5          & 0.195$\pm$0.006          & 4.2$\pm$2.5          \\ \cline{2-8}
	         & RMF-MM     & 0.194$\pm$0.032          & 13.4$\pm$0.6         & 0.145$\pm$0.009          & 154.9$\pm$12.5       & 0.122$\pm$0.004          & 827.7$\pm$116.3      \\ \hline
	LSP      & RMFNL      & \textbf{0.110$\pm$0.004} & 3.2$\pm$0.6          & \textbf{0.073$\pm$0.001} & 5.5$\pm$1.0          & \textbf{0.047$\pm$0.002} & 14.0$\pm$5.2         \\ \hline
	Geman    & RMFNL      & 0.114$\pm$0.004          & 3.1$\pm$0.5          & \textbf{0.073$\pm$0.001} & 8.3$\pm$1.1          & \textbf{0.047$\pm$0.001} & 19.0$\pm$4.9         \\ \hline
	Laplace  & RMFNL      & \textbf{0.111$\pm$0.004} & 2.8$\pm$0.4          & \textbf{0.074$\pm$0.001} & 5.6$\pm$1.0          & \textbf{0.047$\pm$0.002} & 15.9$\pm$6.1         \\ \hline
\end{tabular}
}
\label{tab:synall}
\end{table}

\begin{figure}[ht]
	\centering
	\subfigure[$m=250$.]
	{\includegraphics[width=0.325\columnwidth]{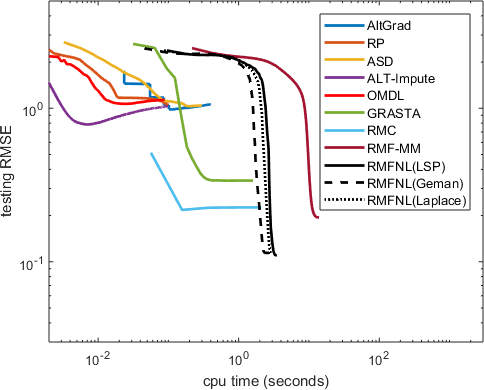}}
	\subfigure[$m=500$.]
	{\includegraphics[width=0.325\columnwidth]{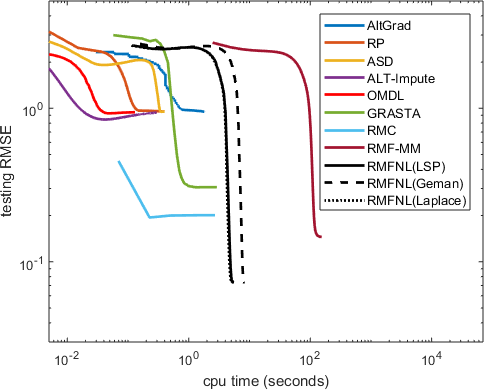}}
	\subfigure[$m=1000$.]
	{\includegraphics[width=0.325\columnwidth]{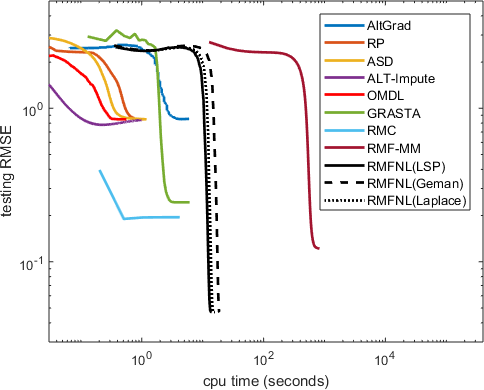}}
	\caption{Convergence of testing RMSE for the various algorithms on
	synthetic data.}
	\label{fig:syncompall}
\end{figure}


\subsection{Robust Collaborative Recommendation}

\label{sec:rec}

In a recommender system,
the love/hate attack changes the ratings
of selected items to the minimum (hate) or maximum (love) \cite{burke2015robust}.
The love/hate attack is very simple, 
but can significantly bias overall prediction.
As no love/hate attack 
data sets
are publicly available,
we follow
\cite{burke2015robust,mobasher2007toward} and manually add permutations.
Experiments are performed on the popular 
\textit{MovieLens}
recommender data sets:\footnote{We have also performed experiments on the larger
\textit{Netflix} and \textit{Yahoo} data sets.  Results are in
Appendix~\ref{app:largerec}.}
\textit{MovieLens-100K},
\textit{MovieLens-1M}, and
\textit{MovieLens-10M}
(Some statistics on these data sets are in Appendix~\ref{app:movielens}).
We
randomly select 3\% of the items from each data set.
For each selected item,
all its observed ratings 
are set 
to either the minimum or maximum with equal possibilities.
$50\%$ of the observed ratings 
are used
for training, $25\%$ for validation, and the rest for testing.
Algorithms in Section~\ref{sec:otheralgs} will be compared.
To reduce statistical variability, results are averaged over five repetitions.
As in Section~\ref{sec:expsyn},
the testing RMSE and CPU time 
are used
for performance evaluation.

Results 
are shown in Table~\ref{tab:mlens}, and
Figure~\ref{fig:movielens} shows convergence of the RMSE.
Again,
RMFNL with different nonconvex losses have similar performance
and achieve the lowest RMSE. 
The 
MF algorithms
are fast, but 
have high RMSEs. 
GRASTA is not stable,
with
large
RMSE 
and variance. 

\begin{table}[ht]
\centering
\caption{Performance on the \textit{MovieLens} data sets.  CPU time is in seconds.  
RMF-MM cannot converge in $10^4$ seconds
on the \textit{MovieLens-1M} and \textit{MovieLens-10M} data sets,
	and thus is not reported.
	}
\scalebox{0.825}{
\begin{tabular}{c |c | c | c | c| c | c | c}
	\hline
	         &            &  \multicolumn{2}{c|}{\textit{MovieLens-100K}}   &    \multicolumn{2}{c}{\textit{MovieLens-1M}}    &    \multicolumn{2}{c}{\textit{MovieLens-10M}}    \\ \cline{3-8}
	  loss   & algorithm  & RMSE                     & CPU time             & RMSE                     & CPU time             & RMSE                     & CPU time              \\ \hline
	$\ell_2$ & AltGrad    & 0.954$\pm$0.004          & 1.0$\pm$0.2          & 0.856$\pm$0.005          & 30.6$\pm$2.5         & 0.872$\pm$0.003          & 1130.4$\pm$9.6        \\ \cline{2-8}
	         & RP         & 0.968$\pm$0.008          & 0.2$\pm$0.1          & 0.867$\pm$0.002          & 4.4$\pm$0.4          & 0.948$\pm$0.011          & 199.9$\pm$39.0        \\ \cline{2-8}
	         & ScaledASD  & 0.951$\pm$0.004          & 0.3$\pm$0.1          & 0.878$\pm$0.003          & 8.7$\pm$0.2          & 0.884$\pm$0.001          & 230.2$\pm$7.7         \\ \cline{2-8}
	         & ALT-Impute & 0.942$\pm$0.021          & 0.2$\pm$0.1          & 0.859$\pm$0.001          & 10.7$\pm$0.2         & 0.872$\pm$0.001          & 198.9$\pm$2.6         \\ \cline{2-8}
	         & OMDL       & 0.958$\pm$0.003          & \textbf{0.1$\pm$0.1} & 0.873$\pm$0.008          & \textbf{2.6$\pm$0.5} & 0.881$\pm$0.003          & \textbf{63.4$\pm$4.2} \\ \hline
	$\ell_1$ & GRASTA     & 1.057$\pm$0.218          & 4.6$\pm$0.3          & 0.842$\pm$0.011          & 31.1$\pm$0.6         & 0.876$\pm$0.047          & 1304.3$\pm$18.0       \\ \cline{2-8}
	         & RMC        & 0.920$\pm$0.001          & 1.4$\pm$0.2          & 0.849$\pm$0.001          & 40.6$\pm$2.2         & 0.855$\pm$0.001          & 526.0$\pm$29.5        \\ \cline{2-8}
	         & RMF-MM     & 0.901$\pm$0.003          & 402.3$\pm$80.0       & ---                      & ---                  & ---                      & ---                   \\ \hline
	  LSP    & RMFNL      & \textbf{0.885$\pm$0.006} & 5.9$\pm$1.5          & \textbf{0.828$\pm$0.001} & 34.9$\pm$1.0         & \textbf{0.817$\pm$0.004} & 1508.2$\pm$69.1       \\ \hline
	 Geman   & RMFNL      & \textbf{0.885$\pm$0.005} & 6.6$\pm$1.2          & \textbf{0.829$\pm$0.005} & 35.3$\pm$0.3         & \textbf{0.817$\pm$0.004} & 1478.5$\pm$72.8       \\ \hline
	Laplace  & RMFNL      & \textbf{0.885$\pm$0.005} & 4.9$\pm$1.1          & \textbf{0.828$\pm$0.001} & 35.1$\pm$0.2         & \textbf{0.817$\pm$0.005} & 1513.4$\pm$12.2       \\ \hline
\end{tabular}
}
\label{tab:mlens}
\end{table}

\begin{figure}[ht]
\centering
\subfigure[\textit{MovieLens-100K}.]
{\includegraphics[width=0.325\columnwidth]{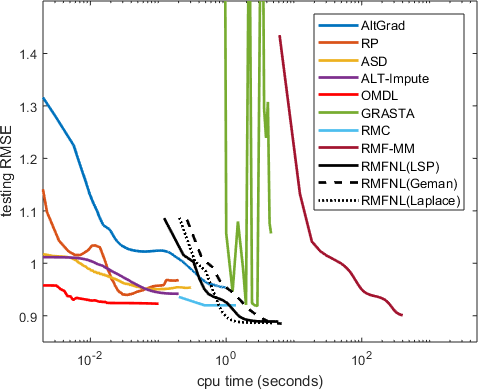}}
\subfigure[\textit{MovieLens-1M}.]
{\includegraphics[width=0.325\columnwidth]{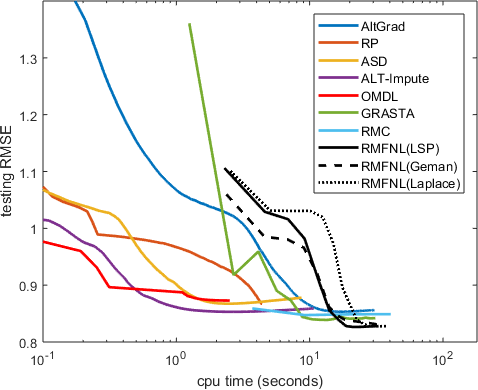}}
\subfigure[\textit{MovieLens-10M}.]
{\includegraphics[width=0.325\columnwidth]{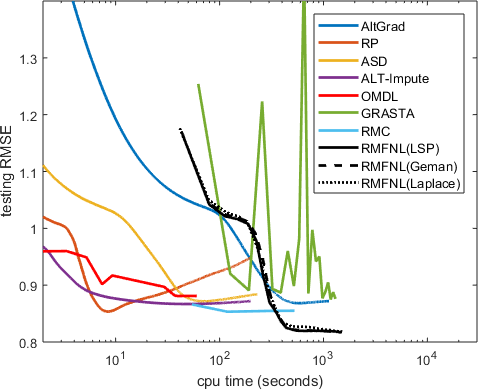}}
\caption{Convergence of testing RMSE on the recommendation data sets.  }
\label{fig:movielens}
\end{figure}


\subsection{Affine Rigid Structure-from-Motion (SfM)}

\label{sec:expSfM}

SfM
reconstructs the 3D scene from sparse feature points tracked in $m$ images of a moving camera
\cite{koenderink1991affine}.
Each feature point is projected to 
every image plane, and 
is thus
represented by a $2 m$-dimensional vector. With
$n$ feature points,
this leads to a 
$2m \times n$
matrix.
Often,
this matrix has missing data (e.g., some feature points may not be always visible)
and outliers (arising from feature mismatch).
We use the Oxford \textit{Dinosaur}
sequence,
which has $36$ images and $4,983$ feature points.
As in \cite{lin2017robust},
we extract three data subsets using feature points 
observed in 
at least $5,6$ and $7$ 
images. 
These are denoted ``D1" (with size 72$\times$932), ``D2"   (72$\times$557) and
``D3"      (72$\times$336).  The fully observed data matrix can be recovered by
rank-$4$ matrix factorization \cite{eriksson2010efficient}, and so we set $r=4$.





We compare RMFNL with RMF-MM and its variant (denoted RMF-MM(heuristic))
described in Section~4.2 of \cite{lin2017robust}.
In this variant,
the diagonal entries of $\Lambda_r$ and $\Lambda_c$ 
are initialized with small values and then 
increased
gradually.
It is claimed 
in \cite{lin2017robust}
that this leads to
faster
convergence. 
However, our experimental results show that this heuristic leads to 
more accurate,
but not faster, results. 
Moreover, the key pitfall of this variant is that 
Proposition~\ref{pr:rmf:surr} and
the convergence guarantee for RMF-MM no longer holds.

For performance evaluation,
as there is no ground-truth,
we follow \cite{lin2017robust} and
use 
the 
(i)
mean absolute error (MAE) 
$\NM{\bar{W} \odot (\bar{U} \bar{V}^{\top} - X)}{1} / \nnz{\bar{W}}$,
where $\bar{U}$ and $\bar{V}$ are outputs from the algorithm,
$X$ is the data matrix with observed positions indicated by the binary $\bar{W}$; and
(ii) CPU time. 
As the various nonconvex penalties have been shown to have similar
performance, we will only report the LSP here.

Results are shown in Table~\ref{tab:rand:dion}.
As can be seen,
RMF-MM(heuristic) obtains a lower MAE than RMF-MM,
but is still outperformed by RMFNL.
RMFNL is the fastest,
though the
speedup is not as significant as in previous sections.
This is because the \textit{Dinosaur} subsets are not very sparse (the percentages
of nonzero entries in
``D1", ``D2" and ``D3" are $17.9\%$, $20.5\%$ and $23.1\%$, respectively).

\begin{table}[ht]
\centering
\caption{Performance on the 
\textit{Dinosaur} 
data subsets.
	CPU time is in seconds.}
\scalebox{0.825}{
\begin{tabular}{c | c | c | c | c | c | c}
\hline
               &             \multicolumn{2}{c|}{D1}             &             \multicolumn{2}{c|}{D2}             &             \multicolumn{2}{c}{D3}              \\ \cline{2-7}
               & MAE                      & CPU time             & MAE                      & CPU time             & MAE                      & CPU time             \\ \hline
RMF-MM(heuristic) & 0.374$\pm$0.031          & 43.9$\pm$3.3         & 0.381$\pm$0.022          & 25.9$\pm$3.1         & 0.382$\pm$0.034          & 10.8$\pm$3.4         \\ \cline{1-7}
  RMF-MM       & 0.442$\pm$0.096          & 26.9$\pm$3.4         & 0.458$\pm$0.043          & 14.9$\pm$2.2         & 0.466$\pm$0.072          & 9.2$\pm$2.1          \\ \hline
   RMFNL       & \textbf{0.323$\pm$0.012} & \textbf{8.3$\pm$1.9} & \textbf{0.332$\pm$0.005} & \textbf{6.8$\pm$1.3} & \textbf{0.316$\pm$0.006} & \textbf{3.4$\pm$1.0} \\ \hline
\end{tabular}
}
\label{tab:rand:dion}
\end{table}


\section{Conclusion}

In this paper,
we 
improved the robustness of 
matrix factorization
by
using a nonconvex loss instead of the commonly used (convex) $\ell_1$ and
$\ell_2$-losses.
Second,
we improved its scalabililty
by exploiting data sparsity (which RMF-MM cannot) and using 
the accelerated proximal gradient algorithm  (which is faster than the commonly
used ADMM).
The space and iteration time complexities are
greatly reduced.
Theoretical analysis 
shows that the proposed RMFNL algorithm 
generates a critical point.
Extensive experiments 
on both synthetic and real-world data sets 
demonstrate that
RMFNL is 
more accurate and more scalable
than the state-of-the-art.

\subsection*{Acknowledgment}
The first author would like to specially thanks for Weiwei Tu and Yuqiang Chen from 4Paradigm Inc.

\cleardoublepage
{
\bibliographystyle{plain}
\bibliography{bib}
}


\appendix

\cleardoublepage


\section{Nonconvex Functions}
\label{app:modmcp}

\subsection{Modification of MCP and SCAD}

For the minimax concave penalty (MCP) \cite{zhang2010nearly}:
\begin{align*}
\phi(|\alpha|)
=
\begin{cases}
|\alpha| - \frac{\alpha^2}{2 \theta} & |\alpha| \le \theta \\
\frac{1}{2}\theta          & |\alpha| > \theta
\end{cases}.
\end{align*}
MCP does not meet Assumption~\ref{ass:phi}
as $\phi$ is not strictly increasing when $|\alpha| > \theta$.
To avoid this problem,
we can simply modify its $\phi$ as
$\tilde{\phi}(|\alpha|) = \phi(|\alpha|) + \delta |\alpha|$,
where $\delta > 0$ is a small constant
(Figure~\ref{fig:mcp}).
The smoothly clipped absolute deviation (SCAD) penalty 
\cite{fan2001variable} 
can be modified in the same way
(Figure~\ref{fig:scad}).

\subsection{Definitions}
\label{sec:defncvxf}

A formal definition of nonconvex functions can be used by RMFML is in Table~\ref{tab:regdef}.

\begin{table}[ht]
\centering
\caption{Example nonconvex regularizers ($\theta > 2$ for SCAD and $\theta > 0$ for others is a constant). Here,
	$\delta > 0$ is a small constant to ensure that
	the $\phi$'s for MCP and SCAD are
	strictly increasing.}
\scalebox{0.82}{
\begin{tabular}{c | c }
	\hline
	& $\phi(|\alpha|)$                  \\ \hline
	Geman penalty 
	& $\frac{|\alpha|}{\theta + |\alpha|}$   \\ \hline
	Laplace penalty 
	&  $1 - \exp\left( -\frac{|\alpha|}{\theta} \right)$                     \\ \hline
	log-sum-penalty (LSP) 
	& $\log\left( 1 + \frac{|\alpha|}{\theta} \right) $ \\ \hline
	minimax concave penalty (MCP)
	& $\begin{cases}
	(1 + \delta)|\alpha| - \frac{\alpha^2}{2 \theta} & \alpha \le  \theta \\
	\frac{1}{2}\theta ^2 + \delta |\alpha|      & \alpha >  \theta
	\end{cases}$                                                                                                                    
	\\ \hline
	smoothly clipped absolute deviation (SCAD) penalty  
	&   $\begin{cases}
	(1 + \delta)|\alpha| & |\alpha| \le 1 \\
	\frac{- \alpha^2 + 2\theta |\alpha| - 1}{2(\theta - 1)} + \delta|\alpha| & 1 < |\alpha| \le \theta \\
	\frac{(1+\theta)}{2} +  \delta|\alpha| & |\alpha| > \theta
	\end{cases}$                                                                                                                    
	\\ \hline
\end{tabular}
}
\label{tab:regdef}
\end{table}


\section{Details of the APG Algorithm}

\subsection{Computing the Gradient}
\label{app:compgrad}

The complete procedure for computing the gradient is shown in Algorithm~\ref{alg:compgrad}.

\begin{algorithm}[ht]
	\caption{Computing $\nabla\mathcal{D}^k(x)$ by exploiting sparsity.}
	\begin{algorithmic}[1]
		\STATE set $X_{i_t j_t} = x_t$ for all $(i_t, j_t) \in \Omega$;
		// i.e., $X = \mathcal{H}_{\Omega}(x)$
		\STATE $Q^k = A_r^k (X V^k) - \lambda (A_r^k U^k)$;
		\STATE obtain 
		$\hat{g}^k
		\in \R^{\nnz{W}}$
		with $\hat{g}^k_t = \sum_{q = 1}^r Q^k_{i_t q} V^k_{j_t q}$;
		\STATE $P^k = A_c^k ( X^{\top} U^k)  - \lambda (A_c^k V^k )$;
		\STATE obtain 
		$\breve{g}^k \in 
		\R^{\nnz{W}}$ with
		$\breve{g}^k_t = \sum_{q = 1}^r U^k_{i_t q} P^k_{j_t q}$; 
		\quad 
		// i.e., $\breve{g}^k = \invH{ U^k (P^k)^{\top} }$
		\RETURN $\hat{g}^k + \breve{g}^k - \invH{ M }$.
	\end{algorithmic}
	\label{alg:compgrad}
\end{algorithm} 

\subsection{Computing the  Objective}
\label{app:compobj}

By the definition of $\mathcal{H}_{\Omega}(x)$,
we construct a sparse matrix $X=\mathcal{H}_{\Omega}(x)$.
We then compute the first term in \eqref{eq:dual} as 
$\frac{1}{2} \NM{P^k \sqrt{ A_r^k }}{F}^2$
where
$P^k = X V^k - \lambda U^k$.
Note that $X$ is sparse with $O(\nnz{W})$ nonzero elements
and $A_r^k$ is a diagonal,
the computation of the first term in \eqref{eq:dual}
takes $O( \nnz{W} r + m r ) $ time,
where $r$ is the number of columns in $U^k$.
Let $y = \invH{M}$.
The second term 
in \eqref{eq:dual} 
can then be computed as $\sum_{i = 1}^{\nnz{W}} x_i y_i $,
which takes $O( \nnz{W} )$ time.
For the last term
in \eqref{eq:dual},
it can be computed similarly as the first term using $O( \nnz{W} r + n r )$ time.
Moreover,
we can see that only $O( \nnz{W} + (m + n)r )$ space is needed.

The whole procedure 
for computing the  objective
is shown in Algorithm~\ref{alg:compobj}.
It  takes 
$O( \nnz{W} + (m + n)r )$ space and
$O( \nnz{W} r + (m + n)r )$ time in total.

\begin{algorithm}[H]
\caption{Computing $\mathcal{D}^k(x)$ by exploiting sparsity.}
\begin{algorithmic}[1]
	\STATE set $X_{i_t j_t} = x_t$ for all $(i_t, j_t) \in \Omega$;
	\STATE $a_1 = \frac{1}{2} \NM{ \sqrt{ A_r^k } P^k }{F}^2$ where $P^k = X V^k - \lambda U^k$;
	
	\STATE $a_2 = \frac{1}{2} \NM{ \sqrt{ A_c^k } Q^k }{F}^2$ where $Q^k = X^{\top} U^k - \lambda V^k$;
	
	\STATE $a_3 = \sum_{i = 1}^{\nnz{W}} x_i y_i $ where $y = \invH{M}$;

	\RETURN $a_1 + a_2 + a_3$.
\end{algorithmic}
\label{alg:compobj}
\end{algorithm} 

\subsection{Computing the Proximal Step}
\label{app:prox}

For the proximal step  with \eqref{eq:dual},
a closed-form solution 
can be obtained by the following Lemma.

\begin{lemma}[\cite{boyd2004convex}]
	\label{lem:prox}
For any given $z$,
	$x^* = \arg\min_{x \in \mathcal{W}^k} \frac{1}{2} \NM{x - z}{F}^2
	= [\sign{z_i} \min( |z_{i}|, (\dot{w}^k_{i})^{-1} )]$.
\end{lemma}


\section{Clarke Subdifferential}
\label{app:clarke}

We first introduce two definitions from \cite{clarke1990optimization}.

\begin{definition}[Clarke subdifferential]
	\label{def:clarke}
	Let $f : \R^{m \times n} \rightarrow \R$ be a 
	\footnote{A function is called locally Lipschitz continuous if for every $X$ in its domain
		there exists a neighborhood $\mathcal{U}$ of $X$ such that $f$ restricted to $\mathcal{U}$ is Lipschitz continuous.}
	locally Lipschitz
	function.
	The {\em Clarke generalized directional derivative} of $f$ at $X$ in the direction of $V$ is:
	\begin{align*}
	f^{\circ}(X, V) 
	\equiv 
	\limsup_{Y \rightarrow X, \lambda \rightarrow 0}
	\frac{1}{\lambda}
	[  f(Y + \lambda V) - f(Y) ].
	\end{align*}
	The 
	{\em Clarke subdifferential}
	of $f$ at $X$ is 
	\begin{align*}
	\partial^{\circ} f(X)
	\equiv
	\lbrace 
	\xi : f^{\circ}(X, V) \ge \Tr{\xi^{\top} V}, 
	\forall \, V \in \R^{m \times n}
	\rbrace.
	\end{align*}
\end{definition}

Note that $f$ in Definition~\ref{def:clarke} can be neither 
convex nor smooth.

\begin{definition}[Critical point]
	\label{def:crti}
	A point $X$ is a critical point of $f$ if it satisfies $0 \in \partial^{\circ} f(X)$.
\end{definition}

\section{Proofs}
\label{app:proof}

\subsection{Preliminaries}

In the section,
we first introduce some Lemmas that will be used later in the proof.  

For a continuous $f$, $\partial^{\circ} {f}$ is the Clarke subdifferential.
The critical points for problem \eqref{eq:mrmf} are defined in the following Lemma.

\begin{lemma} \label{lem:ctri}
	Let $C = M - U V^{\top}$.
	$(U, V)$ is a critical point of \eqref{eq:mrmf} if 
	$0 \in  (W \odot S) V + \lambda U
	\;\text{and}\;
	0 \in (W \odot S)^{\top} U + \lambda V$, 
	where $S_{ij} = \sign{C_{ij}} \phi'( | C_{ij} |  )$ if $C_{ij} \neq
	0$,
	and $S_{ij} \in [ - \phi'(0), \phi'(0) ]$ otherwise.	
\end{lemma}

\begin{proof}
For a nonconvex penalty function $\phi$ satisfying Assumption~\ref{ass:phi},
	from Proposition 5 in \cite{gong2015honor},
	its Clark subdifferential is 
	\begin{align}
	\begin{cases}
	\partial^{\circ} \phi(|\alpha|) = \sign{\alpha} \cdot \phi'(|\alpha|)
	& \text{if}\; \alpha \neq 0
	\\
	\partial^{\circ} \phi(|\alpha|)
	\in \left[ - \phi'(0), \phi'(0) \right]
	& \text{otherwise}
	\end{cases}.
	\label{eq:temp11}
	\end{align}
	By Definition~\ref{def:crti},
	if $(U, V)$ is a critical point of \eqref{eq:mrmf},
	it needs to satisfy
	\begin{align}
	(0, 0) \in \partial^{\circ} \dot{H}(U, V).
	\label{eq:temp12}
	\end{align}
	Combining \eqref{eq:temp11} and \eqref{eq:temp12},
	we obtain the Lemma.
\end{proof}

\begin{lemma}
	\label{lem:realW}
	Define the row sum $\ssum{\dot{W}_{(i, :)}^k} = \sum_{j = 1}^n \dot{W}^k_{ij}$, and
	the column sum
	$\ssum{\dot{W}_{(:, j)}^k} = \sum_{i = 1}^m \dot{W}^k_{ij}$.
	Then,
	$\NM{\dot{W}^k \odot (\bar{U} \bar{V}^{\top})}{1}
	\le \frac{1}{2}\NM{\Lambda_r^k \bar{U}}{F}^2
	+ \frac{1}{2}\NM{\Lambda_c^k \bar{V}}{F}^2$,
	where
	\begin{align*}
	\Lambda_r^k = \Diag{\sqrt{\smash[b]{\ssum{\dot{W}_{(1, :)}^k}}}, \dots ,
		\sqrt{\smash[b]{\ssum{\dot{W}_{(m, :)}^k}}}},
	\end{align*}
	and
	\begin{align*}
	\Lambda_c^k  = \Diag{\sqrt{\smash[b]{\ssum{\dot{W}_{(:, 1)}^k}}}, \dots, \sqrt{\smash[b]{\ssum{\dot{W}_{(:, n)}^k}} }}.
	\end{align*}
	Equality holds iff $( \bar{U}, \bar{V} )  = (0, 0)$.
\end{lemma}

\begin{proof}
	First, we have
	\begin{align}
	\NM{\dot{W}^k \odot (\bar{U} \bar{V}^{\top})}{1}
	& = 
	\left\| 
	\dot{W}^k \odot 
	\begin{bmatrix}
	\underline{u}_1^{\top}  \underline{v}_1 
	\cdots  \underline{u}_1^{\top}  \underline{v}_n
	\\
	\cdots
	\\
	\underline{u}_m^{\top}  \underline{v}_1 
	\cdots  \underline{u}_m^{\top}  \underline{v}_n
	\end{bmatrix}
	\right\|_1
	\notag
	\\ 
	& =
	\sum_{i = 1}^m \sum_{i = 1}^n
	\dot{W}^k_{ij} \left| \underline{u}_i^{\top} \underline{v}_j \right|.
	\label{eq:temp1}
	\end{align}
	where $\underline{u}_i$ is $i$th row in $\bar{U}$ (similar, for $\underline{v}_j$ in $\bar{V}$).
	Then,
	from Cauchy inequality, 
	we have 
	\begin{align*}
	\left| \underline{u}_i^{\top} \underline{v}_j \right|
	\le \NM{\underline{u}_i}{2} \NM{\underline{v}_j}{2}
	\le \frac{1}{2} \left( \NM{\underline{u}_i}{2}^2 + \NM{\underline{v}_j}{2}^2 \right).
	\end{align*}
	Together with \eqref{eq:temp1}, we have
	\begin{align*}
	\NM{\dot{W}^k \odot (\bar{U} \bar{V}^{\top})}{1}
	\le \frac{1}{2} \sum_{i = 1}^m  \sum_{j = 1}^n \dot{W}^k_{ij} 
	\left( \NM{\underline{u}_i}{2}^2 + \NM{\underline{v}_j}{2}^2 \right) 
	= \frac{1}{2}\NM{\Lambda_r \bar{U}}{F}^2
	+ \frac{1}{2}\NM{\Lambda_c \bar{V}}{F}^2,
	\end{align*}
	and the equality holds only when $(\bar{U}, \bar{V}) = (\mathbf{0}, \mathbf{0})$.
\end{proof}


\subsection{Proposition~\ref{pr:bndphi}}
\label{app:pr:bndphi}

\begin{proof}
	Note that $\phi(x)$ is concave on $x \ge 0$.
	For any $y \ge 0$, we have
	\begin{align*}
	\phi(y) \le \phi(x) + (y - x)\phi'(x).
	\end{align*}
	Let $y = |\beta|$ and $x = |\alpha|$.
	We obtain
	\begin{align*}
	\phi(|\beta|) \le \phi(|\alpha|) + (|\beta| - |\alpha|)\phi'(|\alpha|).
	\end{align*}
	As $\phi$ is concave and strictly increasing on $\R^+$,
	equality holds iff $\beta = \pm \alpha$.
\end{proof}


\subsection{Corollary~\ref{cor:loss}}
\label{app:cor:loss}

\begin{proof}
This Corollary can be easily obtained 
(i) using Proposition~\ref{pr:bndphi} 
on the nonconvex loss in \eqref{eq:mrmf};
and (ii) $U = U^k + \bar{U}$ and $V = V^k + \bar{V}$.
\end{proof}


\subsection{Proposition~\ref{pr:bndext}}
\label{app:pr:bndext}

\begin{proof}
	From the Cauchy inequality,
	we have
\begin{eqnarray}
\lefteqn{\NM{\dot{W}^k \odot
( M - (U^k + \bar{U}) \left( V^k + \bar{V} \right)^{\top} )}{1}}
\label{eq:temp2}
\\
& \le &
\NM{\dot{W}^k  \odot  ( M  -  U^k (V^k)^{\top}  -  \bar{U} (V^k)^{\top} 
-  U^k \bar{V}^{\top} )}{1}
+  \NM{\dot{W}^k  \odot  ( \bar{U} \bar{V}^{\top}) }{1}.
\notag
\end{eqnarray}
	For the last term,
	using Lemma~\ref{lem:realW}, we have
	\begin{align}
	\NM{\dot{W}^k \odot ( \bar{U} \bar{V}^{\top})}{1}
	\le \frac{1}{2}
	\left( \NM{\Lambda_r^k \bar{U}}{F}^2
	+ \NM{\Lambda_c^k \bar{V}}{F}^2 \right).
	\label{eq:temp3}
	\end{align}
	Combining \eqref{eq:temp2} and \eqref{eq:temp3},
	we have 
	\begin{align}
	\sum_{i = 1}^m
	\sum_{j = 1}^n
	\dot{W}^k_{ij} 
	\phi\left( \left| M_{ij} - [ U V^{\top} ]_{ij} \right|  \right)
	\le 
	& \NM{\dot{W}^k \odot
		( M - (U^k + \bar{U}) ( V^k + \bar{V} )^{\top} )}{1}
	\notag
	\\
	& +
	\frac{1}{2}
	\left( \NM{\Lambda_r^k \bar{U}}{F}^2
	+ \NM{\Lambda_c^k \bar{V}}{F}^2 \right) + b^k.
	\label{eq:temp4}
	\end{align}
	Adding $ \frac{\lambda}{2} \NM{U^k + \bar{U}}{F}^2
	+ \frac{\lambda}{2} \NM{V^k + \bar{V}}{F}^2
	$
	to both side of \eqref{eq:temp4},
	we obtain the proposition.
	
	\noindent
	Besides,
	from Lemma~\ref{lem:realW},
	the equality in the proposition holds
	only when $(\bar{U}, \bar{V}) = (0, 0)$.
\end{proof}


\subsection{Proposition~\ref{pr:dual}}
\label{app:pr:dual}

\begin{proof}
Using the fact that $\NM{X}{1} = \max_{\NM{Y}{\infty} \le 1} \Tr{X^{\top} Y}$
\cite{boyd2004convex},
where $\NM{Y}{\infty} = \max_{i,j} | Y_{ij} |$ is the $\ell_{\infty}$-norm,
$\mathcal{D}^k(x)$ can be rewritten as
\begin{align*}
\max_{ x \in \mathcal{W}^k } \min_{\bar{U}, \bar{V}}
\mathcal{P}(x, \bar{U}, \bar{V}),
\end{align*}
where 
\begin{align}
\!\!\!\!
\mathcal{P}(x, \bar{U}, \bar{V})
\equiv \; &
\Tr{\mathcal{H}_{\Omega}(x)^{\top} (M  -  \bar{U} (V^k)^{\top}  -  U^k \bar{V}^{\top})}
\label{eq:primaldual}
\\
& +  \frac{\lambda}{2}\NM{U^k + \bar{U}}{F}^2 
+ \frac{1}{2}\NM{\Lambda_r^k  \bar{U}}{F}^2 
+ \frac{\lambda}{2}\NM{V^k + \bar{V}}{F}^2 
+ \frac{1}{2}\NM{\Lambda_c^k  \bar{V}}{F}^2.
\notag
\end{align}
As \eqref{eq:primaldual} is an unconstrained, smooth and convex problem on $\bar{U}$,
the optimal solution is obtained when
$\nabla_{\bar{U}} \mathcal{P}(X, \bar{U}, \bar{V}) = 0$.
Then,
\begin{align}
\bar{U} 
=  A_r^k 
( \mathcal{H}_{\Omega}(x) V^k - \lambda U^k ).
\label{eq:temp18}
\end{align}
Similarly,
we obtain 
\begin{align}
\bar{V} 
=  A_c^k 
( \mathcal{H}_{\Omega}(x)^{\top}  U^k  - \lambda V^k ).
\label{eq:temp19}
\end{align}
Substituting \eqref{eq:temp18} and \eqref{eq:temp19} back into \eqref{eq:primaldual},
we obtain
$\mathcal{D}^k(X)$ in the proposition.
\end{proof}


%


\subsection{Proposition~\ref{pr:limit}}
\label{app:pr:limit}

First,
Proposition~\ref{pr:limit} 
can be elaborated as fololows.

\begin{proposition}
For Algorithm~\ref{alg:rmfnl},
\begin{itemize}
\item[(i).]
$\lbrace ( U^k, V^k ) \rbrace$ is bounded. 

\item[(ii).]
$\{ (U^k, V^k) \}$ has a sufficient decrease on $\dot{H}$, i.e.,
$\dot{H}( U^k, V^k ) - \dot{H}( U^{k + 1}, V^{k + 1} )
\ge \gamma \NM{ U^{k + 1} - U^k }{F}^2 + \gamma \NM{ V^{k + 1} - V^k }{F}^2$,
where $\gamma > 0$ is a constant; and

\item[(iii).]
$\lim_{k \rightarrow \infty} (U^{k + 1} - U^k) = 0$
and $\lim_{k \rightarrow \infty} (V^{k + 1} - V^k) = 0$.
\end{itemize}
\end{proposition}

\begin{proof}
First note that,
\begin{align}
\inf_{U, V} H(U, V) \ge 0
,
\lim\limits_{\substack{\NM{U}{F} \rightarrow \infty \\ \NM{V}{F} \rightarrow \infty}} H(U, V)
= \infty,
\label{eq:temp10}
\end{align}
Then,
the sequence $\{ U^k \}$ and $\{ V^k \}$ is bounded,
and we otbain the result in part (i).

Thus,
there exists a positive constant $c$ such that
\begin{align*}
c_1 \ge | [ U^k (V^k)^{\top} ]_{ij} |,
\quad
\forall i, j, k.
\end{align*}
From Assumption~\ref{ass:phi},
$\phi$ is a strictly increasing function,
thus $\phi' > 0$.
Then,
there exists a positive constant $c_2$ such that
\begin{align*}
\phi'\left( | [ U^k (V^k)^{\top} ]_{ij} | \right) \ge c_2 \equiv \phi'(c_1).
\end{align*}
From Assumption~\ref{ass:wht},
each row and column in $W$ has at least one nonzero element.
By the definition of $\Lambda_r^k$ in Proposition~\ref{pr:bndext},
its diagonal elements is given by
\begin{align*}
\left[ \Lambda_r^k \right]_{ii}
\ge \sqrt{\sum_{j = 1}^n W_{ij} c_2}.
\end{align*}
The same holds for $\Lambda_c^k$.
Thus, there exists a constant $\alpha > 0$
such that all diagonal elements in $\Lambda_r^k$ and $\Lambda_c^k$ are not smaller than it.

As $\left( \bar{U}^k, \bar{V}^k \right)$ is the optimal solution of $\min \dot{F}^k$,
then
\begin{align}
(0, 0) \in \partial \dot{F}^k \left( \bar{U}^k, \bar{V}^k \right).
\label{eq:temp20}
\end{align}
Define
\begin{align*}
\dot{J}^k
( \bar{U}, \bar{V} ) 
\equiv \NM{\dot{W}^k \odot ( M \! - \! U^k (V^k)^{\top} \! - \! \bar{U} (V^k)^{\top} 
\! - \! U^k \bar{V}^{\top} )}{1}
\! + \! \frac{\lambda}{2} \NM{U^k \! + \! \bar{U}}{F}^2
\! + \! \frac{\lambda}{2} \NM{V^k \! + \! \bar{V}}{F}^2 
\! + \! b^k.
\end{align*}
Recall the definition of $\dot{F}^k$.
From \eqref{eq:temp20},
we have
\begin{align*}
\left( G_{\bar{U}^k}, G_{\bar{V}^k} \right) \in \partial \bar{J}^k (\bar{U}^k, \bar{V}^k).
\end{align*}
Thus
\begin{align}
(0, 0) = \left( G_{\bar{U}^k}, G_{\bar{V}^k} \right)
+ \left(  (\Lambda_r^k)^2 \bar{U}, (\Lambda_c^k)^2 \bar{V} \right).
\label{eq:temp5}
\end{align}
Multiplying $(\bar{U}^k, \bar{V}^k)$ on both side of \eqref{eq:temp5},
we have
\begin{align}
0 
= \; & 
\Tr{G_{\bar{U}^k}^{\top} \bar{U}^k} 
+ \Tr{G_{\bar{V}^k}^{\top} \bar{V}^k}
+ \NM{(\Lambda_r^k)^2 \bar{U}}{F}^2
+ \NM{(\Lambda_c^k)^2 \bar{V}}{F}^2.
\label{eq:temp6}
\end{align}
As $\dot{J}^k$ is a convex function,
by the definition of the subgradient,
we have
\begin{eqnarray}
\dot{J}^k(0, 0)
\ge \dot{J}^k(\bar{U}^k, \bar{V}^k)
- \Tr{G_{\bar{U}^k}^{\top} \bar{U}^k} 
- \Tr{G_{\bar{V}^k}^{\top} \bar{V}^k}.
\label{eq:temp7}
\end{eqnarray}
Combining \eqref{eq:temp6} and \eqref{eq:temp7},
we obtain
\begin{align}
\dot{J}^k (0, 0)
& \ge \dot{J}^k(\bar{U}^k, \bar{V}^k)
+ \NM{(\Lambda_r^k)^2 \bar{U}}{F}^2
+ \NM{(\Lambda_c^k)^2 \bar{V}}{F}^2
\notag
\\
& \ge \dot{H}^k(\bar{U}^k, \bar{V}^k)
\! + \! \frac{1}{2} \NM{(\Lambda_r^k)^2 \bar{U}}{F}^2
\! + \! \frac{1}{2} \NM{(\Lambda_c^k)^2 \bar{V}}{F}^2.
\label{eq:temp8}
\end{align}
Note that
\begin{align*}
\dot{J}^k(\mathbf{0}, \mathbf{0}) 
& = H(U^k, V^k),
\\
\dot{H}^k(\bar{U}^k, \bar{V}^k)
& = H(U^{k + 1}, V^{k + 1}),
\end{align*}
and using \eqref{eq:temp8},
we have
\begin{align}
\!\!\!
H(U^k, V^k)
- H(U^{k + 1}, V^{k + 1})
\ge \frac{1}{2} \NM{\Lambda_r^k \bar{U}^k}{F}^2
+ \frac{1}{2} \NM{(\Lambda_c^k)^2 \bar{V}^k}{F}^2
\ge \frac{\alpha}{2} \left(  \NM{\bar{U}^k}{F}^2 + \NM{\bar{V}^k}{F}^2 \right) .
\label{eq:temp9}
\end{align}
Thus, we obtain the result in part (ii) in Proposition~\ref{pr:limit}
(with $\gamma = \alpha / 2$).

Summing all inequalities in \eqref{eq:temp9} from $k = 1$ to $K$, 
we have
\begin{align*}
H(U^1, V^1)
- H(U^{K + 1}, V^{K + 1})
\ge  \sum_{k = 1}^K \frac{\alpha}{2} \NM{ \bar{U}^k }{F}^2
+ \frac{\alpha}{2} \NM{ \bar{V}^k }{F}^2.
\end{align*}
From \eqref{eq:temp10},
we have
\begin{align}
\sum_{k = 1}^{\infty} \NM{\bar{U}^k}{F}^2 < \infty,
\sum_{k = 1}^{\infty} \NM{\bar{V}^k}{F}^2 < \infty,
\label{eq:temp17}
\end{align}
which indicates that
\begin{align*}
\lim\limits_{k \rightarrow \infty}
\NM{\bar{U}^k}{F}^2
& = \lim\limits_{k \rightarrow \infty}
\NM{(U^k - U^{k + 1})}{F}^2
= 0,
\\
\lim\limits_{k \rightarrow \infty}
\NM{\bar{V}^k}{F}^2
& = \lim\limits_{k \rightarrow \infty}
\NM{(V^k - V^{k + 1})}{F}^2
= 0.
\end{align*}
Then, we have the result in part (iii).
\end{proof}


\subsection{Proposition~\ref{pr:direct}}
\label{app:pr:direct}

The following connects the subgradient of surrogate $\dot{F}^k$ to the
Clarke subdifferential of $\dot{H}$.

\begin{proposition} 
	\label{pr:direct}
	(i) $\partial \dot{F}^k(0, 0) = \partial^{\circ} \dot{H}^k(0, 0)$;
	(ii) If $0 \in \partial^{\circ} \dot{H}^k(0, 0)$,
	then $(U^k,  V^k)$ is a critical point of \eqref{eq:mrmf}.
\end{proposition}

\begin{proof}
\textbf{Part (i).}
We prove this by 
the Clark subdifferential of $\dot{H}^k$ and subgradient of $\dot{F}^k$.
\begin{itemize}
	\item \underline{Clark subdifferential of $\dot{H}^k$}:
	Let $C^H = M - U V^{\top}$.
	By the definition of Clark differential, we have  
	\begin{align}
	\partial_{U}^{\circ}
	\dot{H}^k(\bar{U},  \bar{V})
	& =  (W  \odot  S^H) (V^k  +  \bar{V}) 
	+  \lambda (U^k  +  \bar{U}),
	\label{eq:temp13}
	\\ 
	\partial_{V}^{\circ} 
	\dot{H}^k(\bar{U},  \bar{V})
	& =  (W  \odot  S^H)^{\top} (U^k  +  \bar{U}) 
	+  \lambda (V^k  +  \bar{V}),
	\label{eq:temp15}
	\end{align}
	where $S^H_{ij} = 
	\sign{C^H_{ij}} \cdot \phi'\left( \left| C^H_{ij} \right|  \right)$
	if
	$C^H_{ij} \neq 0$,
	and $S^H_{ij} \in \left[ - \phi'(0), \phi'(0) \right]$ otherwise.
	
	\item \underline{Subgradient of $\dot{F}^k$}:
	Let $C^F = M - U^k (V^k)^{\top} - \bar{U}^k (V^k)^{\top} -  U^k (\bar{V}^k)^{\top}$.
	For $\dot{F}^k$, we have
	\begin{align}
	\partial_{U} \dot{F}^k(\bar{U}, \bar{V})
	& = (\dot{W}^k \odot S^F)(V^k + \bar{V}^k)
	+ \lambda (U^k + \bar{U})
	+ (\Lambda_r^k)^2 \bar{U},
	\label{eq:temp14}
	\\
	\partial_{V} \dot{F}^k(\bar{U}, \bar{V})
	& = (\dot{W}^k \odot S^F)^{\top}(U^k + \bar{U}^k)
	+  \lambda (V^k + \bar{V})
	 + (\Lambda_c^k)^2 \bar{V},
	\label{eq:temp16}
	\end{align}
	where $S^F_{ij} = \sign{C^F_{ij}}$ if $C^F_{ij} \neq 0$,
	and $S^F_{ij} \in [-1, 1]$ otherwise.
\end{itemize}
Note that when $\bar{U} = 0$ and $\bar{V} = 0$,
we have $C^H = C^F$.
By the definition of $\dot{W}^k = A^k \odot W$,
we also have $W \odot S^H = \dot{W}^k \odot S^F$.
Finally,
the last term in \eqref{eq:temp14} vanishes to zero as $\bar{U} = 0$.
Thus,
\eqref{eq:temp13} is exactly the same as \eqref{eq:temp14}.
Similarly \eqref{eq:temp15} 
is also the same as
\eqref{eq:temp16}.
As a result,
we have $\partial^{\circ} \dot{F}^k(\mathbf{0}, \mathbf{0}) = \partial^{\circ} \dot{H}^k(0, 0)$.

\noindent
\textbf{Part (ii).}
From the definition of $\dot{H}$ in \eqref{eq:mrmf} and $\dot{H}^k$ in Proposition~\ref{pr:bndext},
we have 
\begin{align*}
\dot{H}^k(\bar{U}, \bar{V})
= \dot{H}(U^k + \bar{U}, V^k + \bar{V}).
\end{align*}
Thus,
if 
$(\mathbf{0}, \mathbf{0}) \in \partial^{\circ} \dot{H}^k(\mathbf{0}, \mathbf{0})$,
we have 
\begin{align*}
(0, 0) \in \partial^{\circ} \dot{H}(U^k, V^k),
\end{align*}
which shows that $(U^k, V^k)$ is a critical point.
\end{proof}


\subsection{Theorem~\ref{thm:critical}}
\label{app:thm:critical}

\begin{proof}
From Proposition~\ref{pr:limit}, we know thata there is at least one limit point for the sequence $\left\lbrace \left( U^k, V^k \right)  \right\rbrace$.
Let $\left\lbrace \left( U^{k_j}, V^{k_j} \right)  \right\rbrace$ be one of its
subsequences,
	and 
	\begin{align*}
	U^* = \lim\limits_{k_j \rightarrow \infty} U^{k_j},
	\quad
	V^* = \lim\limits_{k_j \rightarrow \infty} V^{k_j},
	\end{align*}
	where $(U^*, V^*)$ is a limit point.
	Using Proposition~\ref{pr:direct},
	we have
	\begin{align*}
	\lim\limits_{k_j \rightarrow \infty}
	\partial^{\circ} \dot{F}^{k_j}
	\left( \bar{U}_{k_j}, \bar{V}_{k_j} \right)
	= \lim\limits_{k_j \rightarrow \infty}
	\partial^{\circ} \dot{F}^{k_j}
	\left( 0, 0 \right)
	=  \lim\limits_{k_j \rightarrow \infty}
	\partial^{\circ} \dot{H}^{k_j}
	\left( 0, 0 \right)
	= \partial^{\circ} \dot{H}
	\left( U^*, V^* \right).
	\end{align*}
	Thus,
	$(0, 0)
	\in \partial^{\circ} \dot{H}
	\left( U^*, V^* \right)$,
	which shows that $(U^*, V^*)$ is a critical point (Lemma~\ref{lem:ctri}).
\end{proof}

\section{Additional Materials for the Experiments}


\subsection{Statistics of \textit{MovieLens}.}
\label{app:movielens}

The statistics of \textit{MovieLens} data sets are in following Table~\ref{tab:recSys}.

\begin{table}[ht]
\centering
\caption{\textit{MovieLens} data sets used.}
\small
	\begin{tabular}{c| c | c | c | c}
		\hline
		                        & number of users & number of movies & number of ratings & \% nonzero elements \\ \hline
		\textit{MovieLens-100K} & 943             & 1,682            & 100,000           & 6.30                \\ \hline
		 \textit{MovieLens-1M}  & 6,040           & 3,449            & 999,714           & 4.80                \\ \hline
		\textit{MovieLens-10M}  & 69,878          & 10,677           & 10,000,054        & 1.34                \\ \hline
	\end{tabular}
\label{tab:recSys}
\end{table}

%

\subsection{Experiments on Larger Recommendation Datasets}
\label{app:largerec}

We also perform experiments on two much larger recommendation datasets: \textit{netflix}
(480,189 users, 17,770 items and 100,480,507 ratings) and 
\textit{yahoo} (249,012 users, 296,111 items and 62,551,438 ratings).
The same setup in Section~4.2 is used.
RMC runs out of memory and RMF-MM is too slow.
Thus, they are not compared.
Results are shown on the right 
(CPU time is in minutes). Observations here are the same as those for \textit{MovieLens} data sets.
RMFNL with different nonconvex losses have similar performance
and achieve the lowest RMSE. 
Algorithms for the $\ell_2$-loss have much higher RMSEs than that of $\ell_1$ and RMFNL. 

\begin{table}[H]
	\centering
	\caption{Results on \textit{Netflix} and \textit{Yahoo} datasets.}
	\small
	\begin{tabular}{c |c | c | c | c | c }
		\hline
		         &            & \multicolumn{2}{c|}{\textit{netflix}} & \multicolumn{2}{c}{\textit{yahoo}} \\ \cline{3-6}
		  loss   & algorithm  & RMSE           & time (min)           & RMSE           & time (min)        \\ \hline
		$\ell_2$ & RP         & 0.910          & 142.3                & 0.842          & 105.4             \\ \cline{2-6}
		         & ScaledASD  & 0.918          & 213.9                & 0.864          & 74.2              \\ \cline{2-6}
		         & ALT-Impute & 0.931          & 309.1                & 0.802          & 77.4              \\ \cline{2-6}
		         & OMDL       & 0.923          & \textbf{16.5}        & 0.831          & \textbf{12.3}     \\ \hline
		$\ell_1$ & GRASTA     & 0.857          & 247.3                & 0.751          & 238.5             \\ \hline
		  LSP    & RMFNL      & \textbf{0.805} & 221.0                & \textbf{0.668} & 81.2              \\ \hline
		 Geman   & RMFNL      & \textbf{0.806} & 228.4                & \textbf{0.670} & 98.7              \\ \hline
		Laplace  & RMFNL      & \textbf{0.805} & 216.8                & \textbf{0.669} & 89.2              \\ \hline
	\end{tabular}
\end{table}
\end{document}